\documentclass[a4paper,12pt]{scrartcl}

\usepackage{scrpage2}
\usepackage{wrapfig}
\usepackage{graphicx}
\usepackage{soul}

\usepackage[utf8]{inputenc}
\usepackage[T1]{fontenc}

\usepackage{appendix}

\usepackage{multicol}
\usepackage{needspace}
\usepackage{comment}

\usepackage[normalem]{ulem}

\usepackage{amsmath}
\usepackage{mathtools}
\usepackage{tensor}
\usepackage{leftidx}
\usepackage{amsthm}
\usepackage{enumitem}
\setlist[description]{%
	font={\rmfamily\mdseries \dashuline}, 
}
\usepackage{scalerel}	
\usepackage[varumlaut]{yfonts}
\usepackage{amssymb}
\usepackage{amsfonts}
\usepackage{bbm}
\usepackage{extarrows}
\usepackage{stmaryrd}

\usepackage[mathscr]{euscript}

\usepackage{pgf}
\usepackage{tikz}

\usepackage{enumitem}

\usepackage{fancybox}
\usetikzlibrary{arrows,shapes,trees,matrix,calc}
\usetikzlibrary{decorations.markings}
\usetikzlibrary{matrix}

\usepackage{hyperref}

\usepackage[shortcuts]{extdash}

\allowdisplaybreaks

\newcommand{\kk}{\Bbbk} 

\newcommand{\ZZ}{\mathbb{Z}}

\newcommand{\II}{\mathbb{I}}

\newcommand{\eqbydef}{\mathrel{\overset{\makebox[0pt]{\mbox{\normalfont\tiny def}}}{=}}}
\newcommand{\lto}{\longrightarrow}
\newcommand{\lmapsto}{\longmapsto}

\newcommand{\simto}{\xrightarrow{\sim}}
\newcommand{\lsimto}{\xlongrightarrow{\sim}}

\newcommand{\eps}{\varepsilon}

\newcommand{\id}[1]{\operatorname{id}_{#1}}

\newcommand{\Hom}[2]{ \operatorname{Hom}_{#1}\! \left( #2 \right) }

\newcommand{\End}[2]{\operatorname{End}_{#1}(#2)}

\newcommand{\tr}{\operatorname{tr}}

\newcommand{\ev}{\operatorname{ev}}
\newcommand{\coev}{\operatorname{coev}}

\newcommand{\ot}{\otimes}
\newcommand{\boxt}{\boxtimes}

\DeclareMathOperator*{\bigot}{\scaleobj{1}{\raisebox{2pt}{$\bigotimes$}}}

\newcommand{\lact}{\mathbin{\triangleright}}
\newcommand{\ract}{\triangleleft}
\newcommand{\lract}{\ogreaterthan}

\newcommand{\catcenter}{\operatorname{\mathcal{Z}}}

\newcommand{\reg}{{\operatorname{reg}}}

\newcommand{\tild}[1]{{\widetilde{#1}}}

\newcommand{\mopp}{{\operatorname{mop}}}
\newcommand{\opp}{{\operatorname{op}}} 
\newcommand{\coopp}{{\operatorname{cop}}}
\newcommand{\opcop}{{\opp,\coopp}}

\newcommand{\anti}[2]{{#1^{\langle #2 \rangle}}}

\newcommand{\opj}[2]{\overline{#1}^{#2}}	
\newcommand{\opmod}[2]{\opj{#1}{#2}} 

\newcommand{\dual}[2]{{{#1}^{#2}}}	

\newcommand{\balcat}[3]{\mathcal{Z}_{#2,#3}(#1)}

\newcommand{\cat}[1]{\mathcal{#1}}

\newcommand{\A}{\cat{A}}

\newcommand{\M}{\cat{M}}

\newcommand{\X}{\cat{X}}
\newcommand{\vect}[1]{\operatorname{vect}({#1})}

\newcommand{\lmod}[1]{{#1}\!\operatorname{--mod}}

\newtheoremstyle{indented}{3pt}{3pt}{\addtolength{\leftskip}{5.5em}}{}{\bfseries}{.}{.5em}{}

\theoremstyle{plain}

\newcounter{dummy} 

\newtheorem{theorem}[dummy]{Theorem}
\newtheorem{proposition}[dummy]{Proposition}
\newtheorem{lemma}[dummy]{Lemma}

\theoremstyle{definition}

\newtheorem{definition}[dummy]{Definition}

\theoremstyle{remark}

\newtheorem{remark}[dummy]{Remark}
\newtheorem{remarks}[dummy]{Remarks}

\newtheorem{example}[dummy]{Example}
\newtheorem{examples}[dummy]{Examples}

\theoremstyle{indented}

\usepackage{xfrac}

\usepackage{float}

\usepackage[top=2.0cm, bottom=2.0cm, left=2.0cm, right=2.0cm]{geometry}
\addtokomafont{sectioning}{\rmfamily}

\title{%
	\begin{flushright}
		\vspace{-1.8cm} \normalfont{\small{\textsf{[ZMP-HH/20-1]}}}\\
		\vspace{-0.5cm} \normalfont{\small{\textsf{Hamburger Beiträge zur Mathematik Nr.\! 819}}}\\
	\end{flushright}
	\vspace{0.5cm}
	Defects in Kitaev models and bicomodule algebras
}

\author{%
	Vincent Koppen
	\\[0.25cm]
	{\normalsize\slshape Fachbereich Mathematik, Universit\"{a}t Hamburg, Germany}\\
	\normalsize{\texttt{\href{mailto:vincent.koppen@uni-hamburg.de}{vincent.koppen@uni-hamburg.de}}} 
}

\date{}

\begin{document}
\def\balhopf[#1]#2{ {\color{brown} \mathsf{H^*_{#1,#2}} } }
\newcommand{\balalg}[2]{A_{#1, #2}}
\newcommand{\balalghopf}[3]{{{#1}^*_{#2, #3}}}
\newcommand{\vtxal}[2]{{#1 \lract #2}} 
\newcommand{\statespace}{\mathscr{H}}

\newcommand{\halfedges}[1]{{\Sigma_{#1}^{0.5}}} 
\newcommand{\plaqedges}[1]{{\Sigma_{#1}^{1.5}}} 
\newcommand{\vtxsites}[1]{{\Sigma_{#1}^{\text{sit}}}}
\newcommand{\plaqsites}[1]{{\Sigma_{#1}^{\text{sit}}}}

\newcommand{\edges}[1]{{\Sigma_{#1}^1}}
\newcommand{\plaquettes}[1]{{\Sigma_{#1}^2}}
\newcommand{\vertices}[1]{{\Sigma_{#1}^0}}

\newcommand{\vertexedge}[2]{{\eps({#2})}} 
\newcommand{\plaqedge}[2]{\eps_{#1}({#2})} 


\maketitle

\begin{abstract}
We construct a Kitaev model, consisting of a Hamiltonian which is
the sum of commuting local projectors, for surfaces with boundaries and defects of dimension 0 and 1.
More specifically, we show that one can consider cell decompositions of surfaces whose 2-cells are labeled by semisimple Hopf algebras and 1-cells are labeled by semisimple bicomodule algebras.
We introduce an algebra whose representations label the 0-cells and which reduces to the Drinfeld double of a Hopf algebra in the absence of defects.
In this way we generalize the algebraic structure underlying the standard Kitaev model without defects or boundaries, where all 1-cells and 2-cells are labeled by a single Hopf algebra and where point defects are labeled by representations of its Drinfeld double.
In the standard case, commuting local projectors are constructed using the Haar integral for semisimple Hopf algebras.
A central insight we gain in this paper is that in the presence of defects and boundaries, the suitable generalization of the Haar integral is given by the unique symmetric separability idempotent for a semisimple (bi-)comodule algebra.
\end{abstract}


\section{Introduction}

The Kitaev model has been constructed as a simple model
for topological quantum computing, using a degenerate ground-state space as the code space and a set of commuting local projectors to correct local errors.
It is also known as the quantum double model, surface code or toric code \cite{kitaev, buerschaperEtAl}.
The algebraic input datum for such a construction is, in the simplest situation, a finite-dimensional semisimple complex Hopf algebra \cite{buerschaperEtAl, meusburger}; for the toric code it is the group algebra of the group with two elements.
The ground states of this model are described by a three-dimensional topological field theory of Turaev-Viro type \cite{balKir}, which provides links to quantum topology.

On the other hand, it is interesting to consider such models not just on surfaces, but on surfaces with additional structure.
In terms of physics, we want to allow for defects and boundaries; in mathematical terms, we consider the theories on a suitable class of stratified manifolds called \emph{defect surfaces} in the sense of \cite{fss}, but see also e.g.\! \cite{carqueville-etal}.
(Here we study models on oriented surfaces, whereas in \cite{fss} surfaces with $2$-framings are considered.)
Defects in topological field theories are known to lead to higher-dimensional ground-state spaces and more interesting mapping class group representations of the underlying surfaces on these; see e.g.\! \cite{barkeshli, bilayer, laubscherEtAl}.
This is, in particular, relevant for applications to topological quantum computing, where quantum gates are implemented by mapping class group actions on the code space \cite{freedman-larsen-wang}.
There have been already several approaches to include defects or boundaries in Kitaev models based on group algebras \cite{bravyi-kitaev, bombin+martin-delgado, beigi-shor-whalen, cong-cheng-wang}, but our approach deals with the more general case of semisimple Hopf algebras.

The main result
of this paper is the construction of a Kitaev-type model, consisting of a commuting-projector Hamiltonian, for surfaces with general defects and boundaries, using general Hopf-algebraic and representation-theoretic
input data.

\medskip

For our construction it is necessary to realize the data labeling the defects,
which are known for Turaev-Viro theory in a category-theoretic language,
concretely in Hopf-algebraic and re\-pre\-sen\-tation-theoretic terms.
Specifically, topological field theories of Turaev-Viro type are parameterized by spherical fusion categories \cite{barWest-invariants}.
The data for defects separating two such theories are semisimple bimodule categories \cite{kitaevKong, fsv12, fss}.
The idea for obtaining the data for a Kitaev construction is to invoke Tannaka-Krein duality \cite{deligne}.
It states that a semisimple Hopf algebra is equivalent to specifying a fusion category (the representation category of the Hopf algebra, admitting a canonical spherical structure) together with a monoidal fibre functor valued in finite-dimensional vector spaces (the forgetful functor assigning to a representation its underlying vector space).
This recovers semisimple Hopf algebras as the input datum for the Kitaev models without defects, which we think of here as the labels for the two-dimensional strata of the defect surface.

We extend this idea and employ, for the bimodule categories labelling line defects on the surface in Turaev-Viro theory, the appropriate bimodule versions of fibre functors.
By a bimodule version of Tannaka-Krein duality, which we explain
in Subsection \ref{subsec:bicomodule-algebras},
this realizes these categories as the representation categories of bicomodule algebras over Hopf algebras.
We thus identify bicomodule algebras as the labels for line defects and, as a special case, comodule algebras for boundaries.

Having established the algebraic data for line defects of the surface, we
turn our attention to vertices where such line defects can join.
They are labeled by objects in a category which serves as possible labels for generalized Wilson lines in a corresponding three-dimensional topological field theory, including boundary Wilson lines and Wilson lines at the intersection of surface defects.
This category has been determined as a suitable generalization \cite{fss-trace, fss} of the Drinfeld center for a spherical fusion category, which labels bulk Wilson lines.
Here, in Subsection \ref{subsec:algebra-at-vertex}, this category is realized as a representation category as follows:
For a vertex at which line defects meet, the bicomodule algebras of the line defects and the algebras dual to the Hopf algebras attached to the adjacent two-dimensional strata naturally assemble into an algebra, defined in Definition \ref{def:vertex-algebra}.
This algebra, which in this paper we call \emph{vertex algebra},
generalizes the Drinfeld double of the Hopf algebra, whose representations label point-like excitations in the Kitaev model without defects.
The category of possible labels for such a vertex is then the category of modules over this algebra.
Theorem \ref{thm:gluing-category-as-representation-category_wo-proof}, which we prove in Appendix \ref{app:gluing-category}, states that this category is equivalent to the category of generalized Wilson lines at the intersection of surface defects in a corresponding three-dimensional field theory \cite{fss}.

\medskip 

Furthermore, a choice of cell decomposition on the underlying surface enters the construction of the Kitaev model.
In the standard Kitaev model without defects, every $1$-cell (or \emph{edge}) of the cell decomposition is labeled by a single Hopf algebra.
In our setting this should be seen as the regular bicomodule algebra and we consider this label as the \emph{transparent}
defect.
In our case, edges of the cell decomposition are either transparently labeled or they constitute a non-trivial defect and are labeled by an arbitrary bicomodule algebra.

\medskip

Our construction proceeds in the following steps -- mirroring the construction of the standard Kitaev model without defects, as in e.g.\ \cite{buerschaperEtAl, balKir}.
We first define in Subsection \ref{subsec:state-space} a vector space with local degrees of freedom for each edge and each $0$-cell (or \emph{vertex}) of the cell decomposition.
Then we show in Subsection \ref{subsec:local-reps-on-state-space} that this vector space admits, locally with respect to the cell decomposition, the structure of a bimodule over the algebras attached to the vertices.
This is analogous to the representations of the Drinfeld double for each site, a pair of a vertex and an adjacent 2-cell (or \emph{plaquette}), in the standard Kitaev model without defects.
In this case one then proceeds to use the Haar integral for any semisimple Hopf algebra to define local projectors via these local representations.
One of our main insights, established in Subsection \ref{subsec:symm-sep-idem}, is that, in the presence of defects, the suitable generalization of the Haar integral to semisimple bicomodule algebras is given by the symmetric separability idempotent, see Definition \ref{def:symm-sep-idem}.
The symmetric separability idempotent of a semisimple algebra is unique, which we recall in Proposition \ref{prop:symm-sep-idem-unique}.
Furthermore, we show in Proposition \ref{prop:symm-sep-idem-coinvariant} that for a semisimple (bi-)comodule algebra, the symmetric separability idempotent satisfies a compatibility with the (bi-)comodule structure which generalizes a basic property of the Haar integral of a semisimple Hopf algebra.
In the absence of defects, the symmetric separability idempotent reduces to the Haar integral, as we show in Example \ref{ex:haar-int-as-symm-sep-idem}.

Using such separability idempotents, in Subsection \ref{subsec:vertex-plaquette-operators} we finally construct projectors for each vertex,
as usual called \emph{vertex operators}, and for each plaquette,
as usual called \emph{plaquette operators}.
Our main result, Theorem \ref{thm:vertex-and-plaquette-operators-commute}, is that all vertex operators and plaquette operators commute -- giving rise to an exactly solvable Hamiltonian defined as a sum of commuting projectors, which project to the ground states of the model.

\medskip
Concerning the ground states, our construction can be seen as a concrete representation-theoretic realization of the category-theoretic construction in \cite{fss}.
While in \cite{fss} more general categories than representation categories of Hopf algebras and bicomodule algebras are considered, for us the additional structure of fibre functors on the categories is necessary in order to define a larger vector space which contains the pre-block space and block space
as subspaces.
Moreover, while for the construction in \cite{fss} no semisimplicity is required, in this paper semisimplicity is essential for the construction of commuting local projectors, since we define them in terms of the symmetric separability idempotents.
(See \cite{kms} for some progress on projectors for non-semisimple Hopf algebras.)
Lastly, since semisimple Hopf algebras have an involutive antipode, they have a canonical trivial pivotal structure.
Hence, we can define our model on any surface with orientation.

\begin{section}*{Acknowledgments}
I would like to thank my Ph.D.\ advisor Christoph Schweigert for introducing to me the topic of the paper and for many helpful discussions and valuable advice and feedback.
Furthermore I am grateful to Ehud Meir, Catherine Meusburger and Thomas Voß for fruitful discussions and to Vincentas Mulevičius for help with the figures.
The author is partially supported by the RTG 1670 ``Mathematics inspired by String theory and Quantum Field Theory'' and by the Deutsche Forschungsgemeinschaft (DFG, German Research Foundation) under Germany’s Excellence Strategy -- EXC 2121 ``Quantum Universe'' -- 390833306.
\end{section}

\section{Hopf-algebraic and representation-theoretic labels for surfaces with cell decomposition}

Following the discussion in the introduction, we will explain in this section the input data for our construction.

\medskip

Let $\Sigma$ be a compact oriented surface together with a cell decomposition $(\Sigma^0, \Sigma^1, \Sigma^2)$ with non-empty sets of $0$-cells (or \emph{vertices}), $1$-cells (or \emph{edges}) and $2$-cells (or \emph{plaquettes}), respectively.
This can be thought of as an
embedding of a graph $(\Sigma^0,\Sigma^1)$ into $\Sigma$ such that its complement in $\Sigma$ is the disjoint union of a set $\Sigma^2$ of disks.
Furthermore, let the edges be oriented, i.e.\ there are source and target maps $s, t : \Sigma^1 \lto \Sigma^0$.
If the surface $\Sigma$ has a boundary, then we require that the 1-skeleton of the cell decomposition be contained in the boundary.

For the construction of a Kitaev model one needs as a further input Hopf-alge\-bra\-ic and rep\-re\-sen\-ta\-tion-theo\-re\-tic data labelling the various strata of the cell decomposition.
In the ordinary Kitaev model without defects as in \cite{buerschaperEtAl}, all edges of the cell decomposition are labeled by a single semisimple Hopf algebra $H$, and wherever point-like excitations are considered \cite{balKir}, a vertex is labeled by a
representation of the Drinfeld double $D(H)$ of the Hopf algebra $H$.
In this paper we consider more general labels for the edges, thereby implementing arbitrary line defects (also known as \emph{domain walls} in condensed matter theory) and boundaries in the Kitaev model.
Accordingly we also consider more general labels for vertices, implementing point defects (also known as \emph{point-like excitations}) inside defect lines or boundaries.
For the remainder of this section we will specify the three types of Hopf-algebraic and representation-theoretic data that label the plaquettes, edges and vertices of a cell decomposition.

\subsection{Bicomodule algebras over Hopf algebras for line defects}\label{subsec:bicomodule-algebras}

We fix once and for all an algebraically closed field $\kk$ of characteristic zero.
For the necessary background on Hopf algebras and conventions regarding the notation, see \cite{montgomery, kassel, buerschaperEtAl}.

\begin{definition}\label{def:bicomodule-algebra} \ 
\begin{itemize}
\item
Let $H_1$ and $H_2$ be Hopf algebras over $\kk$.
An \emph{$H_1$-$H_2$-bicomodule algebra $K$} is a $\kk$-algebra $K$ together with an $H_1$-$H_2$-bicomodule structure, i.e.\ with co-associative co-action written in Sweedler notation for comodules as
\begin{align*}
K &\lto H_1 \ot K \ot H_2,	\\
k &\lmapsto k_{(-1)} \ot k_{(0)} \ot k_{(1)},
\end{align*}
which is required to be a morphism of algebras.
If $H_1 = \kk$ or $H_2 = \kk$, then $K$ is just a right $H_2$-comodule or a left $H_1$-comodule algebra, respectively.

A \emph{semisimple bicomodule algebra} is a bicomodule algebra whose underlying algebra is semisimple.

\item 
Let $\Sigma$ be an oriented surface with a cell decomposition with oriented edges.
A \emph{label $H_p$ for a plaquette} $p \in \Sigma^2$ is a semisimple Hopf algebra $H_p$ over $\kk$.

For any edge $e \in \Sigma^1$ let $p_1 \in \Sigma^2$ and $p_2 \in \Sigma^2$ be the labelled plaquettes on the left and on the right of $e$, respectively, with respect to the orientation of $e$ relative to the orientation of $\Sigma$.
Then a \emph{label $K_e$ for the edge} $e$ is a finite-dimensional semisimple $H_{p_1}$-$H_{p_2}$-bicomodule algebra $K_e$ over $\kk$.
\begin{figure}[H]
	\centering
	\def\svgwidth{80mm} 
\begingroup%
  \makeatletter%
  \providecommand\color[2][]{%
    \errmessage{(Inkscape) Color is used for the text in Inkscape, but the package 'color.sty' is not loaded}%
    \renewcommand\color[2][]{}%
  }%
  \providecommand\transparent[1]{%
    \errmessage{(Inkscape) Transparency is used (non-zero) for the text in Inkscape, but the package 'transparent.sty' is not loaded}%
    \renewcommand\transparent[1]{}%
  }%
  \providecommand\rotatebox[2]{#2}%
  \newcommand*\fsize{\dimexpr\f@size pt\relax}%
  \newcommand*\lineheight[1]{\fontsize{\fsize}{#1\fsize}\selectfont}%
  \ifx\svgwidth\undefined%
    \setlength{\unitlength}{459.02310674bp}%
    \ifx\svgscale\undefined%
      \relax%
    \else%
      \setlength{\unitlength}{\unitlength * \real{\svgscale}}%
    \fi%
  \else%
    \setlength{\unitlength}{\svgwidth}%
  \fi%
  \global\let\svgwidth\undefined%
  \global\let\svgscale\undefined%
  \makeatother%
  \begin{picture}(1,0.55024897)%
    \lineheight{1}%
    \setlength\tabcolsep{0pt}%
    \put(0,0){\includegraphics[width=\unitlength,page=1]{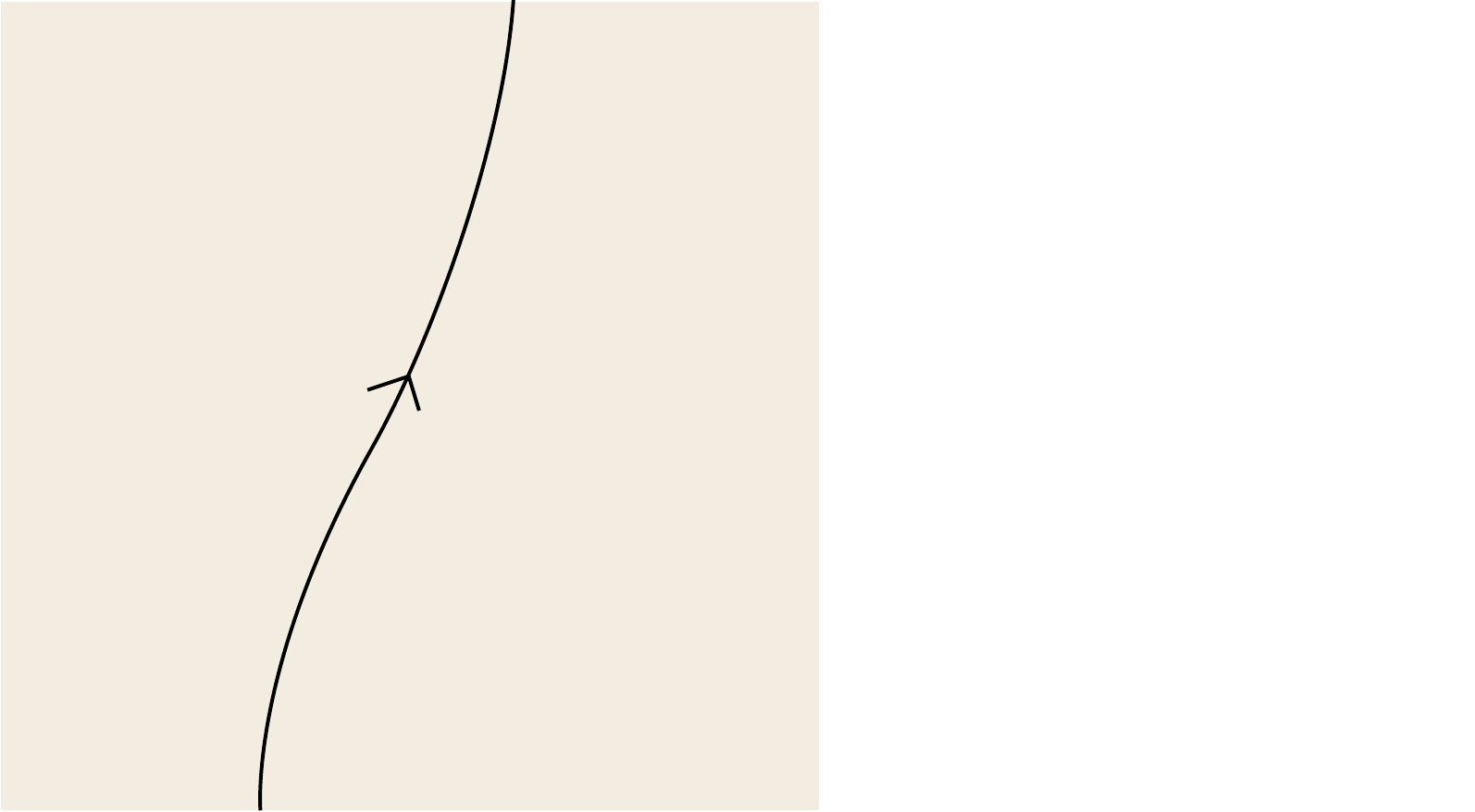}}%
    \put(0.22231,0.10550756){\color[rgb]{0,0,0}\makebox(0,0)[t]{\lineheight{1.25}\smash{\begin{tabular}[t]{c}$e$\end{tabular}}}}%
    \put(0.11867377,0.31884875){\color[rgb]{0,0,0}\makebox(0,0)[t]{\lineheight{1.25}\smash{\begin{tabular}[t]{c}$p_1$\end{tabular}}}}%
    \put(0.41931223,0.16432818){\color[rgb]{0,0,0}\makebox(0,0)[t]{\lineheight{1.25}\smash{\begin{tabular}[t]{c}$p_2$\end{tabular}}}}%
    \put(0.68498008,0.48677787){\color[rgb]{0,0,0}\makebox(0,0)[lt]{\lineheight{1.25}\smash{\begin{tabular}[t]{l}$H_{p_1}$ : Hopf algebra\end{tabular}}}}%
    \put(0.68293936,0.12849174){\color[rgb]{0,0,0}\makebox(0,0)[lt]{\lineheight{1.25}\smash{\begin{tabular}[t]{l}$H_{p_2}$ : Hopf algebra\end{tabular}}}}%
    \put(0,0){\includegraphics[width=\unitlength,page=2]{fig__edge_label.pdf}}%
    \put(0.68466158,0.30741148){\color[rgb]{0,0,0}\makebox(0,0)[lt]{\lineheight{1.25}\smash{\begin{tabular}[t]{l}$K_e$ : $H_{p_1}$-$H_{p_2}$-bicomodule algebra\end{tabular}}}}%
    \put(0,0){\includegraphics[width=\unitlength,page=3]{fig__edge_label.pdf}}%
  \end{picture}%
\endgroup%

	\caption{An edge $e$ and the adjacent plaquettes $p_1$ and $p_2$ with their algebraic data. The two arrows denote the orientations of the edge and, respectively, of the surface $\Sigma$ into which the edge is embedded.}
	\label{fig:edge-labelling}
\end{figure}
If the edge $e$ lies in the boundary of $\Sigma$ and hence only has a plaquette $p$ on one side (left or right), then $K_e$ is just a left or right $H_p$-comodule algebra, respectively.
\end{itemize}
\end{definition}
\begin{examples} \ \begin{enumerate}[ref={\theremark.(\arabic*)}]
\item \label{ex:reg-bicomodule-alg}
Let $H$ be a Hopf algebra.
The \emph{regular} $H$-bicomodule algebra is the algebra underlying the Hopf algebra $H$ together with left and right co-action given by the co-multiplication of $H$.
Note that the regular $H$-bicomodule algebra is semisimple if and only if the Hopf algebra $H$ is semisimple, since both are defined by the semisimplicity of the underlying algebra.
\item
Let $G$ be a finite group and $\kk G$ its group algebra, which has a basis $(b_g)_{g\in G}$ parametrized by $G$ and multiplication induced by the group multiplication.
$\kk G$ is a semisimple Hopf algebra with comultiplication given by the diagonal map $b_g \mapsto b_g \ot b_g$ for all $g\in G$.
Further, let $U \subseteq G$ be a subgroup and $\zeta \in Z^2(U,\kk^\times)$ a group $2$-cocycle.
Then the cocycle-twisted group algebra $\kk U_\zeta$ with multiplication $b_u \cdot b_v := \zeta(u,v) b_{uv}$ for all $u,v \in U$ is a $\kk G$-comodule algebra with co-action given by the diagonal map $b_u \mapsto b_u \ot b_u$.
\end{enumerate}
\end{examples}

\subsubsection{Tannaka-Krein duality: a category-theoretic motivation for bicomodule algebras} \label{subsubsec:tannaka-duality}

Let us explain the emergence of bicomodule algebras from the point of view of Tannaka-Krein duality, as outlined in the Introduction.
We thereby relate the algebraic input data for our construction, as defined in Definition \ref{def:bicomodule-algebra}, to the category-theoretic data for the state-sum construction of a modular functor in \cite{fss}.
For the relevant category-theoretic notions and background, see e.g.\ \cite{egno}.

First of all, for a finite-dimensional Hopf algebra $H$ over $\kk$, it is well known that the category $\lmod{H}$ of finite-dimensional left $H$-modules is a finite $\kk$-linear tensor category.
This tensor category comes equipped with a forgetful functor $\lmod{H} \lto \vect{\kk}$ into the tensor category of finite-dimensional vector spaces.
The forgetful functor is monoidal, exact and faithful.

In fact, it is known \cite{egno} that the datum of a finite-dimensional Hopf algebra $H$ over $\kk$ is equivalent to the datum of a finite $\kk$-linear tensor category $\A$ together with a monoidal fiber functor $\omega : \A \lto \vect{\kk}$, i.e.\! an exact and faithful $\kk$-linear tensor functor to the category of finite-dimensional vector spaces.
More precisely, the Hopf algebra $H$ can be reconstructed as the algebra of natural endo-transformations of the fiber functor $\omega$ and the tensor structure on the fiber functor $\omega$ induces the additional coalgebra structure on the algebra $H$, such that $\A \cong \lmod{H}$ as tensor categories.

We extend this idea to bimodule categories as follows.
For a finite-dimensional $H_1$-$H_2$-bicomodule algebra $K$ for Hopf algebras $H_1$ and $H_2$, the category $\lmod{K}$ has the structure of an $(\lmod{H_1})$-$(\lmod{H_2})$-bimodule category in a natural way.
Indeed, if $X_1$ is an $H_1$-module, $X_2$ is an $H_2$-module and $M$ is a $K$-module, then $X_1 \lact M \ract X_2 := X_1 \ot_\kk M \ot_\kk X_2$ becomes a $K$-module by pulling back the natural $(H_1 \ot K \ot H_2)$-action on it along the co-action map $K \lto H_1 \ot K \ot H_2$ that belongs to $K$.

On the other hand, let $(\A_1, \omega_1 : \A_1 \lto \vect{\kk})$ and $(\A_2, \omega_2 : \A_2 \lto \vect{\kk})$ be finite $\kk$-linear tensor categories together with monoidal fiber functors.
Consider $\vect{\kk}$ as an $\A_1$-$\A_2$-bimodule category via the monoidal functors $\omega_1$ and $\omega_2$.
Let $\M$ be a finite $\kk$-linear $\A_1$-$\A_2$-bimodule category.
Then we define a \emph{bimodule fiber functor} $\omega : \M \lto \vect{\kk}$ for $\M$ to be an exact and faithful $\A_1$-$\A_2$-bimodule functor from $\M$ to the category of finite-dimensional vector spaces.
Let $H_1$ and $H_2$ be the corresponding finite-dimensional Hopf algebras over $\kk$ corresponding to $(\A_1, \omega_1)$ and $(\A_2, \omega_2)$.
Then, by the same argument as for tensor categories \emph{mutatis mutandis}, the bimodule structure on the fiber functor $\omega$ induces the structure of an $H_1$-$H_2$-bicomodule algebra $K$ on the algebra of natural endo-transformations of $\omega$, such that $\omega$ induces an equivalence of bimodule categories $\M \cong \lmod{K}$.

\medskip

\noindent Hence, we conclude that bicomodule algebras emerge naturally as the algebraic input data for Kitaev models, if one follows the following idea in order to obtain concrete Hopf-algebraic data:
Take the category-theoretic data underlying the corresponding topological field theories or modular functors, which are tensor categories and bimodule categories \cite{fss,kitaevKong}, and equip them with fiber functors of the appropriate type.

\subsection{Algebraic structure at half-edges and sites}
\label{subsec:half-edges-and-sites}

It remains to determine the possible labels for the vertices of the cell decomposition.
This is the content of Subsection \ref{subsec:algebra-at-vertex}.
Before that, in this Subsection \ref{subsec:half-edges-and-sites}, we first introduce suitable notation and terminology in order to extract and conveniently speak about the combinatorial information contained in the cell decomposition.

\medskip

Fix a vertex $v \in \Sigma^0$.
Then let $\halfedges{v}$ be the set of \emph{half-edges incident to $v$}.
This is the set of incidences of an edge with the given vertex $v \in \Sigma^0$.
(A loop at $v$ yields two half-edges incident to $v$.)
Note that we have a map $\halfedges{v} \lto \Sigma^1$, assigning to any half-edge its underlying edge, which is in general not injective due to the possible existence of loops.
We will denote by $\edges{v}$ its image in $\Sigma^1$, that is the set of edges starting or ending at the given vertex $v$.

We will say that $e \in \halfedges{v}$ is \emph{directed away from} $v \in \Sigma^0$ if $v = s(e)$ and, that $e \in \halfedges{v}$ is \emph{directed towards} $v \in \Sigma^0$ if $v = t(e)$.
Then for any half-edge $e \in \halfedges{v}$ incident to the vertex $v \in \Sigma^0$, let the sign $\vertexedge{v}{e} \in \{ +1, -1 \}$ be
positive if the half-edge $e \in \halfedges{v}$ is directed away from the vertex $v$:
\begin{figure}[H]
\begin{minipage}[c]{0.25\textwidth}
	\centering
	\begin{flushright}
	\def\svgwidth{15mm}
	\input{fig__half-edge-pos__fixed.pdf_tex}
	\end{flushright}
\end{minipage}\hfill
\begin{minipage}[c]{0.73\textwidth}
	\caption{A half-edge $e\in\halfedges{v}$ incident to $v$ with sign $\vertexedge{v}{e} := +1$} \label{fig:positive-half-edge}
\end{minipage}
\end{figure}
and negative if $e \in \halfedges{v}$ is directed towards $v$:
\begin{figure}[H]
	\begin{minipage}[c]{0.25\textwidth}
		\centering
		\begin{flushright}
		\def\svgwidth{15mm}
		\input{fig__half-edge-neg__fixed.pdf_tex}
		\end{flushright}
	\end{minipage}\hfill
	\begin{minipage}[c]{0.73\textwidth}
		\caption{A half-edge $e\in\halfedges{v}$ incident to $v$ with sign $\vertexedge{v}{e} := -1$} \label{fig:negative-half-edge}
	\end{minipage}
\end{figure}
Let $p\in\Sigma^2$ be the plaquette on the left of the half-edge $e\in\halfedges{v}$, as seen from the vertex $v\in\Sigma^0$, and let $p' \in \Sigma^2$ be the plaquette on the right, as in Figure \ref{fig:half-edge-label}.
\begin{figure}[H]
	\centering
	\def\svgwidth{40mm}
	\input{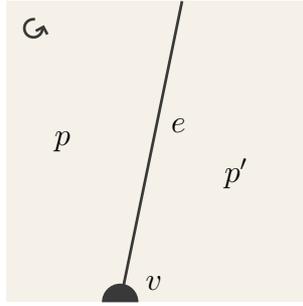}
	\caption{A half-edge $e$ at $v$ with neighboring plaquettes $p$ and $p'$}
	\label{fig:half-edge-label}
\end{figure}
What we have not represented in the figure is that the half-edge $e$ comes with an orientation, expressed by the sign $\eps := \vertexedge{v}{e}$.
By our assignment of labels, if the half-edge $e$ is directed away from the vertex $v$, i.e.\ $\eps = +1$, then it is labeled with an $H_p$-$H_{p'}$-bicomodule algebra $K_e$, with co-action written in Sweedler notation for comodules:
\[ \begin{array}{rll}
K_e &\lto &H_p \ot K_e \ot H_{p'} \\
k &\lmapsto &k_{(-1)} \ot k_{(0)} \ot k_{(1)}
\end{array} \bigg\} \text{ if } \vertexedge{v}{e} = +1. \]
If, on the other hand, the half-edge $e$ points towards $v$, that is $\eps = -1$, then $K_e$ is an $H_{p'}$-$H_p$-bicomodule algebra:
\[ \begin{array}{rll}
K_e &\lto &H_{p'} \ot K_e \ot H_p \\
k &\lmapsto &k_{(-1)} \ot k_{(0)} \ot k_{(1)}
\end{array} \bigg\} \text{ if } \vertexedge{v}{e} = -1. \]
We shall introduce notation that allows us to treat both cases $\eps=+1$ and $\eps=-1$ at once.
Let
\begin{align*}
K_e^{+1} &:= K_e	\\
K_e^{-1} &:= K_e^{\opp},
\end{align*}
where $K_e^\opp$ is the algebra with opposite multiplication.
Moreover, let
\begin{align*}
H_p^{+1} &:= H_p,	\\
H_p^{-1} &:= H_p^{\opp\coopp},
\end{align*}
where $H_p^{\opp\coopp}$ is the Hopf algebra with opposite multiplication and opposite comultiplication.
If $K_e$ is a left (or right, respectively) $H_p$-comodule algebra, then $K_e^{-1}$ is canonically a left (or right, respectively) $H_p^{-1}$-comodule algebra.

Hence, in both above cases we can write that $K_e^\eps$ is an $H_p^\eps$-$H_{p'}^\eps$-bicomodule algebra, with co-action in Sweedler notation:
\begin{align*}
	K_e^\eps &\lto H_p^\eps \ot K_e^\eps \ot H_{p'}^\eps, \\
	k &\lmapsto k_{(-\eps)} \ot k_{(0)} \ot k_{(\eps)}.
\end{align*}

\medskip
Denote by $\vtxsites{v}$ the set of \emph{sites incident to $v$}. These are incidences of a plaquette $p \in \Sigma^2$ with the given vertex $v \in \Sigma^0$.
(Note that a plaquette $p \in \Sigma^2$ can have two separate incidences with the vertex $v$.
This happens when an edge in its boundary is a loop.)
Dually, for a plaquette $p\in\Sigma^2$ denote by $\plaqsites{p}$ the set of \emph{sites incident to $p$}. These are incidences of a vertex $v \in \Sigma^0$ with the given plaquette $p$.
It is justified to use the name \emph{site} for both notions:
To any site $p \in\vtxsites{v}$ at a vertex $v\in\Sigma^0$ corresponds a unique site $\tild{v} \in \plaqsites{p}$ with underlying vertex $v$ at the plaquette that underlies the site $p\in\vtxsites{v}$.

Now let $p \in \vtxsites{v}$ be such a site at the vertex $v\in \Sigma^0$.
There is a half-edge $e'_p \in \halfedges{v}$ bounding $p$ on the left as seen from the vertex $v$ and there is a half-edge $e_p \in \halfedges{v}$ bounding $p$ on the right.
For an example consider Figure \ref{fig:site}.
\begin{figure}[h!]
	\centering
	\def\svgwidth{40mm}
	\input{fig__site__fixed.pdf_tex}
	\caption{A site $p\in\vtxsites{v}$ with neighboring half-edges $e'_p$ and $e_p$.}
	\label{fig:site}
\end{figure}

Then, in consideration of the respective signs $\eps := \vertexedge{v}{e_p}$ and $\eps' := \vertexedge{v}{e'_p}$ of the half-edges $e_p$ and $e'_p$, we have by our assignment of labels that $K_{e'_p}^{\eps'}$ is a right $H_p^{\eps'}$-comodule algebra and that $K_{e_p}^\eps$ is a left $H_p^\eps$-comodule algebra.
In other words, we have a left $((H_p^{\eps'})^\coopp \ot H_p^\eps)$-comodule structure on the algebra
\begin{equation}\label{eq:site-bicomodule-algebra}
K_{\{e_p,e'_p\}} := \bigot_{e \in \{e_p,e'_p\}\subseteq\halfedges{v}} K_e^{\vertexedge{v}{e}} =
\begin{cases}
K_{e'_p}^{\eps'} \ot K_{e_p}^\eps, & e_p\neq e'_p\in\halfedges{v} \\
K_{e_p}^\eps, & e_p = e'_p \in\halfedges{v}
\end{cases}.
\end{equation}

Next we introduce, for a fixed site $p \in \vtxsites{v}$, a canonical left $((H_p^{\eps'})^\coopp \ot H_p^\eps)$-module algebra,
which we think of as associated to the site $p$:
\begin{definition}	\label{def:balancing-algebra-hopf}
Let $v\in\Sigma^0$ be a vertex and $p\in\vtxsites{v}$ a site at $v$ with neighboring half-edges $e_p, e'_p\in\halfedges{v}$ with signs $\eps, \eps' \in \{+1, -1\}$ as before.

The \emph{$\eps'$-$\eps$-balancing algebra $H_p^*$}, or more explicitly $(H_p)_{(\eps',\eps)}^*$, is the left $((H_p^{\eps'})^\coopp \ot H_p^\eps)$-module algebra, whose underlying $\kk$-algebra is the dual algebra of the Hopf algebra $H_p$,
with the following action.
\begin{align*}
((H_p^{\eps'})^\coopp \ot H_p^\eps) \ot H_p^* &\lto H_p^*, \\
a' \ot a \ot f &\lmapsto f(\anti{a'}{-\eps'}\cdot ? \cdot \anti{a}{\eps}),
\end{align*}
where
\[ \anti{a}{\eps} :=
\begin{Bmatrix}
a, & \eps = +1 \\
S(a), & \eps = -1
\end{Bmatrix} \quad \text{ for all } a \in H_p \]
and where $S : H_p \lto H_p$ denotes the antipode.
\end{definition}

Together, the $((H_p^{\eps'})^\coopp \ot H_p^\eps)$-comodule algebra $K_{\{e_p,e'_p\}}$, associated to the half-edges $e_p\in \halfedges{v}$ and $e'_p\in \halfedges{v}$, and the $((H_p^{\eps'})^\coopp \ot H_p^\eps)$-module algebra $H_p^*$, associated to the site $p \in \vtxsites{v}$ situated between the edges $e_p$ and $e'_p$, can be \emph{coupled} into a single $\kk$-algebra,
denoted by
\begin{equation} \label{eq:site-algebra}
H_p^* \lract K_{\{e_p,e'_p\}}
\end{equation}
which has underlying vector space $H_p^* \ot K_{\{e_p,e'_p\}}$ and which is an instance of the following general construction.
For related constructions see \cite{montgomery}.
\begin{definition} \label{def:crossed-product-algebra}
Let $H$ be a Hopf algebra over $\kk$, let $A$ be a left $H$-module algebra and let $K$ be a left $H$-comodule algebra.
Then the
\emph{crossed product algebra} \( \vtxal{A}{K} \) is the $\kk$-algebra with underlying vector space $A\ot K$ and multiplication
\[ (a \ot k) \cdot (a' \ot k') := a (k_{(-1)}.a') \ot k_{(0)} k' \quad\text{ for } (a\ot k), (a'\ot k') \in A \ot K . \]
\end{definition}
In particular, the algebra $\vtxal{H_p^*}{K_{\{e_p,e'_p\}}}$ contains $H_p^*$ and $K_{\{e_p,e'_p\}}$ as subalgebras and the commutation relation between these is
\begin{equation} \label{eq:straightening-formula}
k \cdot f = f(\anti{{k}_{(\eps')}}{-\eps'}\cdot ? \cdot \anti{k_{(-\eps)}}{\eps}) \cdot k_{(0)} \quad\forall f \in H_p^*, k \in K_{\{e_p,e'_p\}},
\end{equation}
the so-called \emph{straightening formula}.
This generalizes the straightening formula of the Drinfeld double of a Hopf algebra, see Example \ref{ex:drinfeld-double-as-vertex-algebra}.

\subsection{Vertex algebras and their representations as labels for vertices} \label{subsec:algebra-at-vertex}

In this subsection we introduce, for each vertex $v\in\Sigma^0$, an algebra over $\kk$, which is constructed from the algebraic labelling in the neighbourhood of the vertex $v$.
The representations of this algebra will serve as possible labels for the vertex $v$.
In a corresponding three-dimensional topological field theory these are the possible labels for generalized Wilson lines.

Let us collect the algebras $K_e^{\vertexedge{v}{e}}$ of all half-edges $e \in \halfedges{v}$ incident to the vertex $v \in \Sigma^0$ into a tensor product
\[ K_{\halfedges{v}} := \bigotimes_{e \in \halfedges{v}} K_e^\vertexedge{v}{e} . \]
With the notation of the previous subsection, for each site $p\in\vtxsites{v}$ with neighboring half-edges $e_p$ and $e'_p$ as in Figure \ref{fig:site}, the algebra $K_{\{e'_p,e_p\}}$ is a left comodule over
\[ \big( H_p^{\vertexedge{v}{e'_p}} \big)^\coopp \ot H_p^{\vertexedge{v}{e_p}}	.	\]
This trivially extends to an $((H_p^{\vertexedge{v}{e_p'}})^\coopp \ot H_p^{\vertexedge{v}{e_p}})$-comodule structure on the tensor product $K_{\halfedges{v}}$ of $K_{\{e,e'\}}$ with the algebras attached to the remaining half-edges in $\halfedges{v}$.
The co-actions on $K_{\halfedges{v}}$ for different sites commute with each other, because they come from the bicomodule structures of the tensor factors $(K_e)_{e\in\halfedges{v}}$, making $K_{\halfedges{v}}$ a left comodule algebra over the tensor product of Hopf algebras
\begin{equation} \label{eq:tensor-product-of-balancing-algebras-around-vertex}
\bigot_{p\in\vtxsites{v}} \big( H_p^{\vertexedge{v}{e'_p}} \big)^\coopp \ot H_p^{\vertexedge{v}{e_p}} .
\end{equation}
For each site $p \in \vtxsites{v}$ we want to \emph{couple} the balancing algebra $H_p^*$ to $K_{\halfedges{v}}$, similarly as in \eqref{eq:site-algebra}.
For this we collect the balancing algebras of the sites around the vertex $v$ into a tensor product
\begin{equation*}
H_{\vtxsites{v}}^* := \bigotimes_{p \in \vtxsites{v}} H_p^*. 
\end{equation*}
This is a left module algebra over the tensor product of Hopf algebras  as in \eqref{eq:tensor-product-of-balancing-algebras-around-vertex}.
Now we have all the ingredients to introduce:
\medskip
\begin{definition} \label{def:vertex-algebra}
Let $v\in \Sigma^0$.
The \emph{$\kk$-algebra $C_v$ associated to the vertex $v$}, or \emph{vertex algebra}, is defined as follows.
For any site $p \in \vtxsites{v}$ denote by $e'_p$ and $e_p \in\halfedges{v}$ the half-edges bounding $p$ on the left and on the right, respectively, from the perspective of the vertex $v$, as illustrated in Figure \ref{fig:site}.
Then let
\begin{equation*}\label{eq:vertex-algebra}
C_v := \vtxal{H_{\vtxsites{v}}^*}{K_{\halfedges{v}}}
		= \vtxal{\left( \bigot_{p \in \vtxsites{v}} H_p^* \right)}{\left( \bigot_{e \in \halfedges{v}} K_e^\vertexedge{v}{e} \right)}
\end{equation*}
be the crossed product algebra, as introduced in Definition \ref{def:crossed-product-algebra},
for the left module algebra $H_{\vtxsites{v}}^*$ and the left comodule algebra $K_{\halfedges{v}}$ over the tensor product \eqref{eq:tensor-product-of-balancing-algebras-around-vertex} of Hopf algebras.
\end{definition}
\medskip
In particular, the algebra contains $H_{\vtxsites{v}}^* = \otimes_{p \in \vtxsites{v}} H_p^*$ and $K_{\halfedges{v}} = \otimes_{e \in \halfedges{v}} K_e^\vertexedge{v}{e}$
as subalgebras and, for each site $p'\in\vtxsites{v}$, we have the commutation relation \eqref{eq:straightening-formula}; so in other words,
\begin{equation} \label{eq:site-subalgebra-of-vertex-algebra}
\vtxal{H_{p'}^*}{K_{\{e_{p'}, e'_{p'} \}}} \subseteq \vtxal{\left( \bigot_{p \in \vtxsites{v}} H_p^* \right)}{\left( \bigot_{e \in \halfedges{v}} K_e^\vertexedge{v}{e} \right)} = C_v
\end{equation}
is a subalgebra of $C_v$.

\begin{example} \label{ex:drinfeld-double-as-vertex-algebra}
Let us consider the situation where the vertex $v\in\Sigma^0$ has precisely one half-edge $e$, which is directed away from the vertex and which is labeled by the regular $H$-bicomodule algebra $H$, the \emph{transparent} label.
\begin{figure}[H]
\begin{minipage}[c]{0.35\textwidth}
\begin{flushright}
\def\svgwidth{40mm}
\input{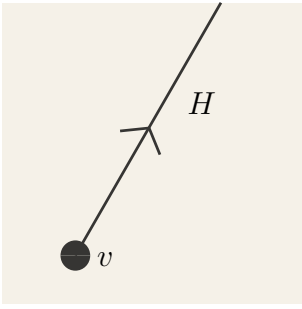}
\end{flushright}
\end{minipage}
\hfill
\begin{minipage}[c]{0.55\textwidth}
	\begin{flushleft}
		\caption{A vertex $v$ with a single half-edge transparently labeled by $H$; \newline the associated algebra $C_v$ is the Drinfeld double $D(H)$}
		\label{fig:transparent-half-edge}
	\end{flushleft}
\end{minipage}
\end{figure}
\noindent Then for the algebra $C_v$ at the vertex $v$ we have $\vtxal{H_{\vtxsites{v}}^*}{K_{\halfedges{v}}} = \vtxal{H^*}{H}$ and the commutation relation \eqref{eq:straightening-formula} gives
\begin{equation} \label{eq:straightening-formula-standard}
h \cdot f = f(S(h_{(3)}) \cdot ? \cdot h_{(1)}) \cdot h_{(2)} .
\end{equation}
This is precisely the so-called \emph{straightening formula} of the Drinfeld double $D(H)$ of a semisimple Hopf algebra $H$ \cite{kassel}.
In the Kitaev model without defects as in \cite{buerschaperEtAl, balKir}, representations of the Drinfeld double $D(H)$ label point-like excitations.
\end{example}
~ \medskip
 
\noindent Up to this point we have explained how, for a given vertex $v \in \Sigma^0$, the algebraic labelling of the edges and plaquettes and the combinatorial structure of the cell decomposition around that vertex gives rise to the $\kk$-algebra $C_v = \vtxal{H_{\vtxsites{v}}^*}{K_{\halfedges{v}}}$.
\begin{definition} \label{def:rep-category-of-labels-at-vertex}
We declare the \emph{category of possible labels} for a vertex $v \in \Sigma^0$ for the Kitaev construction to be the $\kk$-linear category $\lmod{C_v}$ of finite-dimensional left modules over the $\kk$-algebra $C_v$.
\end{definition}

Indeed, in \cite{fss}, the category-theoretic data assigned to a vertex $v\in\Sigma^0$ is as follows.
In the language of \cite{fss}, a vertex $v$ corresponds to a boundary circle $\mathbb{L}_v$ with marked points on which defect lines end.
A $2$-cell $p\in\Sigma^2$ is labelled by a finite tensor category; in our context this is the representation category $\lmod{H_p}$ of a finite-dimensional Hopf algebra $H_p$.
An edge $e \in \Sigma^1$ is labelled by a finite bimodule category; in our context this is the representation category $\lmod{K_e}$ of a bicomodule algebra $K_e$.
Then according to \cite[Definitions~3.4~and~3.9]{fss} the category of possible labels of a vertex $v \in \Sigma^0$ is given by the category $\textup{T}(\mathbb{L}_v)$ of so-called balancings on the Deligne tensor product $\boxt_{e\in\halfedges{v}} (\lmod{K_e^{\vertexedge{v}{e}}})$ of the bimodule categories labelling the half-edges around the vertex $v$.
\begin{theorem}	\label{thm:gluing-category-as-representation-category_wo-proof}
Let $v\in\Sigma^0$.
There is a canonical equivalence of $\kk$-linear categories
\begin{equation*}
\textup{T}(\mathbb{L}_v) \cong \lmod{C_v}
\end{equation*}
between the category assigned by the modular functor $\textup{T}$, constructed in \cite{fss}, to the circle $\mathbb{L}_v$ with marked points corresponding to the half-edges incident to $v$ and the representation category of the algebra $C_v$.
\end{theorem}
\begin{proof}
The proof requires the introduction of significant additional notation and is therefore relegated to the Appendix \ref{app:gluing-category}, see Theorem \ref{thm:gluing-category-as-representation-category}.
\end{proof}

Furthermore, in the case that the edges incident to the vertex $v$ are labeled transparently by a single Hopf algebra $H$ seen as the regular $H$-bicomodule algebra, then the category $\lmod{C_v}$ is equivalent to the Drinfeld center $Z(\lmod{H})$ \cite[Remarks~3.5~(iii)~and~5.23]{fss}, which is equivalent to the category of representations of the Drinfeld double $D(H)$.
These are also the possible labels for point-like excitations in the Kitaev model without defects, cf.\ \cite{balKir}.

\section{Construction of a Kitaev model with defects}
\label{subsec:state-space}

Having specified in the preceding subsections the algebraic input data for the Kitaev model and, in particular, having identified the possible labels for vertices, we are now in a position to construct, for any oriented surface $\Sigma$ with labeled cell decomposition, the vector space and local projectors of the model. 

We recall that we have for each plaquette $p\in\Sigma^2$ a semisimple Hopf algebra $H_p$, for each edge $e\in\Sigma^1$ a semisimple algebra $K_e$ with a compatible bicomodule structure over the Hopf algebras of the incident plaquettes, and for each vertex $v\in\Sigma^0$ a left module $Z_v$ over the algebra $C_v = \vtxal{H_{\vtxsites{v}}^*}{K_{\halfedges{v}}}$, introduced in Definition \ref{def:vertex-algebra}.
We abbreviate
\begin{align*}
 K_{\Sigma^1} &:= \bigot_{e\in\Sigma^1} K_e, \\
 Z_{\Sigma^0} &:= \bigot_{v\in\Sigma^0} Z_v,
\end{align*}
for the tensor products as vector spaces over $\kk$.
More precisely, $K_{\Sigma^1}$ enters our construction of the local projectors and the Hamiltonian of the model not only as a vector space, but together with its structure as the regular $(\bigot_{e\in\Sigma^1} K_e)$-bimodule and its various
co-actions with respect to the Hopf algebras labeling the plaquettes. 
Similarly, we will regard $Z_{\Sigma^0}$ together with its $C_v$-module structure for every vertex $v\in \Sigma^0$.

 \medskip 

The first thing we construct is the vector space, 
on which subsequently the commuting local projectors and the Hamiltonian will be defined.
\begin{definition} \label{def:statespace}
The \emph{state space} assigned to an oriented surface $\Sigma$ with labeled cell decomposition as above is the vector space
\begin{equation} \label{eq:the-vector-space}
\mathscr{H} := \Hom{\kk}{K_{\Sigma^1}, Z_{\Sigma^0}} = (\bigot_{e\in\Sigma^1} K_e^*) \ot (\bigot_{v\in\Sigma^0} Z_v).
\end{equation}
We refer to a tensor factor associated to an edge $e$ or to a vertex $v$ as a \emph{local degree of freedom} associated to $e$ or $v$, respectively.
\end{definition}
\begin{remarks} \ \begin{enumerate}[ref={\theremark.(\arabic*)}]
\item
In the standard Kitaev construction without defects, the vector space is a tensor product of copies of a single Hopf algebra $H$ for every edge, which we interpret in our context as the regular bicomodule algebra over $H$ (the transparent labeling), and for every vertex the dual vector space of a module over $D(H)$ \cite{buerschaperEtAl, balKir}.
In our construction, we instead consider a module over the algebra $C_v$ for every vertex $v\in\Sigma^0$ and the vector space duals of the bicomodule algebras for the edges.
This dual version will make it easier to compare our ground-state spaces with the block spaces of \cite{fss}.
\item
In order to define the state space $\statespace$ we are implicitly using that we do not only have the categories $(\lmod{K_e})_{e\in\Sigma^1}$ and $(\lmod{H_p})_{p\in\Sigma^2}$ as algebraic input data, but we also have the algebras $(K_e)_{e\in\Sigma^1}$ and $(H_p)_{p\in\Sigma^2}$, of which they are the representation categories.
In other words, we need fibre functors on the categories $(\lmod{K_e})_{e\in\Sigma^1}$ and $(\lmod{H_p})_{p\in\Sigma^2}$ to the category of vector spaces in order to define $\mathscr{H}$ as a space of $\kk$-linear homomorphisms.
\item
Note that we are only defining a vector space over $\kk$, and not a Hilbert space, i.e.\ we do not consider a scalar product here.
Accordingly, when we speak of \emph{projectors} on this vector space we always mean idempotent endomorphisms.
By a \emph{Hamiltonian} we mean a diagonalizable endomorphism.
\end{enumerate}
\end{remarks}

\subsection{Local representations of the vertex algebras on the state space}
\label{subsec:local-reps-on-state-space}

Next, we exhibit on the vector space $\statespace$ a natural $C_v$-bimodule structure for each vertex $v \in \Sigma^0$, that is \emph{local} in the sense that it acts non-trivially only on the local degrees of freedom in a neighborhood of the vertex $v \in \Sigma^0$.
This is analogous to the existence of local actions of the Drinfeld double $D(H)$ on the state space in the ordinary Kitaev model without defects for a semisimple Hopf algebra $H$ \cite{buerschaperEtAl, balKir}.
In our construction, however, the algebras $C_v$ are in general not Hopf algebras and we only obtain \emph{bi}module structures on $\statespace$.
(A $C_v$-bimodule structure is equivalent to a left $(C_v \ot C_v^\opp)$-action, where $C_v^\opp$ has the opposite multiplication of $C_v$.
Whenever $C_v$ is a Hopf algebra, such as $D(H)$, any $C_v$-bimodule structure can be pulled back to a left $C_v$-action via the algebra map $(\id{}\ot S)\circ\Delta : C_v \to C_v \ot C_v^\opp$, using the co-multiplication $\Delta$ and the antipode $S$ of the Hopf algebra.)

\medskip

Let $v \in \Sigma^0$ be any vertex.
Recall from Subsection \ref{subsec:algebra-at-vertex} that the algebra $$C_v = \vtxal{H_{\vtxsites{v}}^*}{K_{\halfedges{v}}}$$ is a crossed product of $H_{\vtxsites{v}}^*$ and $K_{\halfedges{v}}$ and contains these as subalgebras, and that $$H_{\vtxsites{v}}^* = \bigot_{p \in \vtxsites{v}} H_p^*$$ is the tensor product of the algebras $H_p^*$ for each site $p\in\vtxsites{v}$.
A $C_v$-bimodule structure on $\statespace$ is therefore fully determined by a $K_\halfedges{v}$-bimodule structure and $H_p^*$-bimodule structures for each site $p \in \vtxsites{v}$, provided that for each $p\in \vtxsites{v}$ the left and right actions of $K_\halfedges{v}$ and $H_p^*$ each satisfy the straightening formula \eqref{eq:straightening-formula} of the crossed product algebra $\vtxal{H_p^*}{K_{\halfedges{v}}}$, which we prove in Theorem \ref{thm:straightening-formula}.

We start by exhibiting a
$K_{\halfedges{v}}$-bimodule structure on the vector space $\mathscr{H}$.
This is the analogon of the action of the Hopf algebra $H$ for every vertex in the ordinary Kitaev model for a semisimple Hopf algebra $H$.
\begin{definition} \label{def:vertex-bimodule}
Let $v \in \Sigma^0$.
The \emph{$K_{\halfedges{v}}$-bimodule structure
on $\statespace$}
\begin{align*}
\tild{A}_v : K_{\halfedges{v}} \ot K_{\halfedges{v}}^\opp \ot \statespace &\lto \statespace, 
\end{align*}
is defined on the vector space of linear maps $\mathscr{H} = \Hom{\kk}{K_{\Sigma^1}, Z_{\Sigma^0}}$
in the standard way by pre-composing with the left action on $K_{\Sigma^1}$ and post-composing with the left action on $Z_{\Sigma^0}$, which are defined as follows:
\begin{itemize}
\item
Firstly, the vector space $K_{\Sigma^1}$ becomes a left $K_{\halfedges{v}}$-module as follows.
Restrict the regular $K_{\Sigma^1}$-bimodule structure of $K_{\Sigma^1}$, seen as a left $(K_{\Sigma^1} \ot K_{\Sigma^1}^\opp)$-action,
to the subalgebra $K_\halfedges{v} \subseteq K_{\Sigma^1} \ot K_{\Sigma^1}^\opp$.
\item
Secondly, the vector space $Z_{\Sigma^0}$ becomes a left $K_\halfedges{v}$-module as follows.
Restrict the given $C_v$-module structure on $Z_v$ to the subalgebra $K_{\halfedges{v}} \subseteq \bigot_{v\in\Sigma^0} (\vtxal{H_{\vtxsites{v}}^*}{K_{\halfedges{v}}}) = C_v$ and extend the action trivially to the vector space $Z_{\Sigma^0} = Z_v \ot \bigot_{w\in\Sigma^0\setminus \{v\}} Z_w$.
\end{itemize}
\end{definition}

\medskip

Next we will exhibit, for any site $p\in\vtxsites{v}$ incident to a vertex $v\in\Sigma^0$,  an $H_p^*$-bimodule structure on $\statespace$.

Recall that $\plaqsites{p}$ denotes the set of incidences of a vertex with a given plaquette $p$ (which we also call \emph{sites}) and denote by $\plaqedges{p}$ the set of incidences of an edge with the given plaquette $p$ (which we call \emph{plaquette edges}). 
We consider their union $\plaqsites{p} \cup \Sigma_p^1$ together with a cyclic order on it, given by the clockwise direction along the boundary of $p$ with respect to the orientation of $\Sigma$, as illustrated in Figure \ref{fig:cyclic-order-of-plaquette-boundary}
\begin{figure}[H]
\begin{minipage}[c]{0.35\textwidth}
\centering
\begin{flushright}
	\def\svgwidth{40mm}
	\input{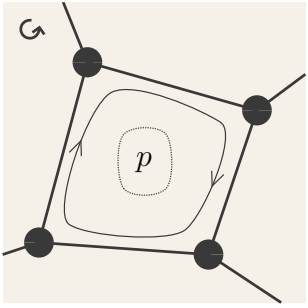}
\end{flushright}
\end{minipage}\hfill
\begin{minipage}[c]{0.6\textwidth}
	\caption{Cyclic order on the set $\plaqsites{p} \cup \Sigma_p^1$ of sites and plaquette edges of a plaquette $p$}
\label{fig:cyclic-order-of-plaquette-boundary}
\end{minipage}
\end{figure}
Furthermore, for any plaquette edge $e \in \Sigma_p^1$ at the plaquette $p$, let the sign $\plaqedge{p}{e} \in \{ +1, -1 \}$ be
positive if the plaquette edge $e \in \Sigma_p^1$ is clockwise directed around the plaquette $p$:
\begin{figure}[H]
	\begin{minipage}[c]{0.3\textwidth}
		\centering
		\begin{flushright}
		\def\svgwidth{20mm}
		\input{fig__plaquette_edge_directed_positively__fixed.pdf_tex}
		\end{flushright}
	\end{minipage}\hfill
	\begin{minipage}[c]{0.7\textwidth}
		\caption{A plaquette edge $e$ with sign $\plaqedge{p}{e} := +1$} \label{fig:positive-plaquette-edge}
	\end{minipage}
\end{figure}
and negative if $e \in \Sigma_p^1$ is directed counter-clockwise around $p$:
\begin{figure}[H]
	\begin{minipage}[c]{0.3\textwidth}
		\centering
		\begin{flushright}
\def\svgwidth{20mm}
\input{fig__plaquette_edge_directed_negatively__fixed.pdf_tex}
		\end{flushright}
	\end{minipage}\hfill
	\begin{minipage}[c]{0.7\textwidth}
		\caption{A plaquette edge $e$ with sign $\plaqedge{p}{e} := -1$} \label{fig:negative-plaquette-edge}
	\end{minipage}
\end{figure}

Recall that, attached to each plaquette $p\in\Sigma^2$, there is a Hopf algebra $H_p$.
Now, depending on choice of a site $v \in \plaqsites{p}$ at $p$,
we define an $H_p^*$-bimodule structure on the vector space $\statespace$.
This is the analogon of the action of the dual Hopf algebra $H^*$ for every site in the ordinary Kitaev model for a semisimple Hopf algebra $H$.

\begin{definition} \label{def:plaquette-bimodule}
Let $p\in \Sigma^2$.
We define, for each site $v \in \plaqsites{p}$, the \emph{$H_p^*$-bimodule structure on $\statespace$}, or left action of the enveloping algebra $H_p^* \ot (H_p^*)^\opp$,
\begin{align*}
	\tild{B}_{(p,v)} : H_p^* \ot (H_p^*)^\opp \ot \statespace &\lto \statespace ,
\end{align*}
by the following left and right $H_p^*$-actions on $\statespace$.
\begin{itemize}
\item
We start by declaring that $H_p^*$ acts from the left on $\statespace = (\bigot_{e\in\Sigma^1} K_e^*) \ot (\bigot_{w\in\Sigma^0} Z_w)$ by the action of $H_p^* \subseteq \vtxal{H_{\vtxsites{v}}^*}{K_{\halfedges{v}}}$ on the $(\vtxal{H_{\vtxsites{v}}^*}{K_{\halfedges{v}}})$-module $Z_v$ and by acting as the identity on the remaining tensor factors of $\statespace$.
\item
For the right action of $H_p^*$ on $\statespace$,
we use the total order on the set $(\plaqsites{p} \cup \plaqedges{p}) \setminus \{ v \}$ starting right after $v \in \plaqsites{p}$ in $\plaqsites{p} \cup \plaqedges{p}$ with respect to the cyclic order declared above, given by the clockwise direction around the plaquette $p$.
We first exhibit individual right $H_p^*$-actions on the tensor factors of $(\bigot_{e\in\Sigma_p^1} K_e^*) \ot (\bigot_{w\in\Sigma_p^0\setminus\{v\}} Z_w)$:
\begin{itemize}
	\item
	For any $e \in \plaqedges{p}$, the vector space $K_e^*$ becomes a right $H_p^*$-module as follows.
	$K_e$ is a right $H_p^{\plaqedge{p}{e}}$-comodule and, hence, a left ${(H_p^*)}^{\plaqedge{p}{e}}$-module.
	Thus the vector space dual $K_e^*$ becomes a right ${(H_p^*)}^{\plaqedge{p}{e}}$-module, and finally, by pulling back along the algebra isomorphism $\anti{?}{\plaqedge{p}{e}} : H_p^* \to {H_p^*}^{\plaqedge{p}{e}}$, a right $H_p^*$-module.
	
	Recall that $\anti{?}{+1} \eqbydef \id{H_p^*}$ and $\anti{?}{-1} \eqbydef S$, the antipode of $H_p^*$.
	Explicitly, this right $H_p^*$-action is given by
	\begin{align*}
	K_e^* \ot H_p^* &\lto K_e^*,	\\
	\varphi \ot f &\lmapsto \left( k \mapsto \varphi\left( k_{(0)} f\left( \anti{k_{(\plaqedge{p}{e})}}{\plaqedge{p}{e}} \right) \right) \right) .
	\end{align*}
	\item
	For any $w\in \plaqsites{p} \setminus \{v\}$, the vector space $Z_w$ becomes a right $H_p^*$-module as follows.
	The $(\vtxal{H_{\Sigma_w^2}^*}{K_{\Sigma_w^1}})$-module $Z_w$ comes with a left $H_p^*$-action since $H_p^* \subseteq \vtxal{H_{\Sigma_w^2}^*}{K_{\Sigma_w^1}}$ is a subalgebra.
	We let $H_p^*$ act on $Z_w$ from the right by pulling back this left action along the antipode $\anti{?}{-1} = S : H_p^* \to H_p^*$.
\end{itemize}
Then we declare $H_p^*$ to act from the right on the tensor product $(\bigot_{e\in\Sigma_p^1} K_e^*) \ot (\bigot_{w\in\Sigma_p^0\setminus\{v\}} Z_w)$ by applying the co-multiplication on $H_p^*$ suitably many times and then acting individually on the tensor factors in the sequence given by the image of the clockwise linear order that we have prescribed on the set $(\plaqsites{p} \cup \plaqedges{p}) \setminus \{ v \}$ under the map $(\plaqsites{p} \cup \plaqedges{p}) \setminus \{ v \} \to (\Sigma_p^0 \cup \Sigma_p^1) \setminus \{ v \}$ that assigns to a site its underlying vertex and to a plaquette edge its underlying edge.
Finally, this gives a right $H_p^*$-action on $\statespace = (\bigot_{e\in\Sigma^1} K_e^*) \ot (\bigot_{w\in\Sigma^0} Z_w)$ by acting with the identity on all remaining tensor factors.
\end{itemize}
\end{definition}

\medskip

So far we have defined, in Definitions \ref{def:vertex-bimodule} and \ref{def:plaquette-bimodule}, on the vector space $\statespace$ an $K_{\halfedges{v}}$-bimodule structure $\tild{A}_v$ for each vertex $v \in \Sigma^0$ and an $H_p^*$-bimodule structure $\tild{B}_{(p,v)}$ for each site $p \in \vtxsites{v}$.
These are analogous to the actions of the Hopf algebra $H$ and the dual Hopf algebra $H^*$ defined for each site in the ordinary Kitaev model without defects.
Just as the latter are shown to interact with each other non-trivially, giving a representation of the Drinfeld double $D(H)$ at each site \cite{buerschaperEtAl}, we will now proceed to study how the bimodule structures $\tild{A}_v$ and $\tild{B}_{(p,v')}$ of $K_{\halfedges{v}}$ and $H_p^*$ for various $v$ and $(p,v')$ interact with each other.

In order to simplify the proof we will make a certain regularity assumption on the cell decomposition of the surface $\Sigma$: 
We call a cell decomposition \emph{regular} if it has no looping edges, i.e.\ there is no edge which has the same source vertex as target vertex and if the Poincaré-dual cell decomposition also has no looping edges, i.e.\ in the original cell decomposition there is no plaquette that has two incidences with one and the same edge (on its two sides).

\begin{theorem}  \label{thm:straightening-formula}
Let $\statespace$ be the vector space defined in Definition \ref{def:statespace} for an oriented surface $\Sigma$ with a labelled cell decomposition.
Recall from Definitions \ref{def:vertex-bimodule} and \ref{def:plaquette-bimodule} the $K_\halfedges{v}$-bimodule structure $\tild{A}_v$ on $\statespace$ for every vertex $v\in\Sigma^0$, and the $H_p^*$-bimodule structure $\tild{B}_{(p,v)}$ on $\statespace$ for every plaquette $p \in \Sigma^2$ together with incident site $v' \in \plaqsites{p}$. 
Then
\begin{itemize}
\item
For any pair of vertices $v_1 \neq v_2 \in \Sigma^0$, the actions $\tild{A}_{v_1}$ and $\tild{A}_{v_2}$ commute with each other.
\item
For any pair of sites $(p_1\in\Sigma^2, v_1\in\plaqsites{p_1})$ and $(p_2\in\Sigma^2, v_2\in\plaqsites{p_2})$ such that $p_1\neq p_2$, the actions $\tild{B}_{(p_1,v_1)}$ and $\tild{B}_{(p_2,v_2)}$ commute with each other.
\item
Assume that the cell decomposition of $\Sigma$ is regular.
For any site $(p\in\Sigma^2, v\in\plaqsites{p})$,
the actions $\tild{A}_{v}$ and $\tild{B}_{(p,v)}$ compose to give on $\statespace$ a bimodule structure over the crossed product algebra $\vtxal{H_{(p,v)}^*}{K_{\halfedges{v}}}$, 
\begin{align*}
\tild{B}_{(p,v)} \tild{A}_v : H_p^* \ot K_{\halfedges{v}} \ot (H_p^* \ot K_{\halfedges{v}})^\opp \ot \statespace &\lto \statespace,	\\
f \ot k \ot f' \ot k' \ot x &\lmapsto \tild{B}_{(p,v)}^{f\ot f'}  \tild{A}_v^{k\ot k'} (x)  .
\end{align*}
\end{itemize}
\end{theorem}
\begin{proof} \
\begin{itemize}
\item
The left $K_{\halfedges{v_1}}$- and $K_{\halfedges{v_2}}$-actions act as the identity on all tensor factors of $ \statespace $ except on $Z_{v_1}$ and $Z_{v_2}$, respectively.
It is thus clear that they commute for $v_1 \neq v_2$.

The right $K_{\halfedges{v_1}}$- and $K_{\halfedges{v_2}}$-actions only have a common tensor factor on which they do not act by the identity for every edge $e\in\Sigma^1$ that joins the vertices $v_1$ and $v_2$.
Such an edge is directed away from one of the vertices and directed towards the other.
Hence, the action for one of the vertices comes from left multiplication of $K_e$ and the other one from right multiplication, so they commute.
\item
The left $H_{p_1}^*$- and $H_{p_2}^*$-actions act as the identity on all tensor factors of $ \statespace$ except on $Z_{v_1}$ and $Z_{v_2}$, respectively.
It is thus clear that they commute for $v_1 \neq v_2$.
In the remaining case $v_1=v_2=:v$, $H_{p_1}^*$ and $H_{p_2}^*$ are commuting subalgebras in $C_v$.
Since their actions on $Z_v$ are by Definition \ref{def:plaquette-bimodule} the restrictions of the $C_v$-action that $Z_v$ comes with, they must therefore commute.

The right $H_{p_1}^*$- and $H_{p_2}^*$-actions only have a common tensor factor on which they do not act by the identity for every vertex $v\in\Sigma^0$ and for every edge $e\in\Sigma^1$ that lies in the boundaries of both plaquettes $p_1$ and $p_2$.
For any such vertex $v$, the two actions come from the $(\vtxal{H_{\vtxsites{v}}^*}{K_{\halfedges{v}}})$-action on $Z_v$ restricted to the two subalgebras $H_{p_1}^*$ and $H_{p_2}^*$, respectively.
These subalgebras commute inside $\vtxal{H_{\vtxsites{v}}^*}{K_{\halfedges{v}}}$, therefore showing the claim.
\item
The left $K_{\halfedges{v}}$- and $H_p^*$-actions on $\statespace$ are simply the restrictions of the left $C_v$-action on $Z_v$ to $K_{\halfedges{v}}$ and $H_p^*$, respectively, and the identity on all other tensor factors of $\statespace$.
Hence, by construction they satisfy the commutation relations of the crossed product algebra $\vtxal{H_{p}^*}{K_{\halfedges{v}}} \subseteq C_v$, see also \eqref{eq:site-subalgebra-of-vertex-algebra}.

The right $K_{\halfedges{v}}$- and $H_p^*$-actions on $\statespace$ are non-trivial only on the tensor factors $\bigot_{e\in\Sigma_v^1} K_e^*$ and $(\bigot_{e\in\Sigma_p^1} K_e^*) \ot (\bigot_{w\in\Sigma_p^0\setminus\{v\}} Z_w)$, respectively.
We can therefore restrict our attention to the vector space $(\bigot_{e\in\Sigma_v^1\cup\Sigma_p^1} K_e^*) \ot (\bigot_{w\in\Sigma_p^0\setminus\{v\}} Z_w)$, on which $K_{\halfedges{v}}$ and $H_p^*$ act from the right.

For convenience, for the remainder of the proof we now switch to the dual vector space $(\bigot_{e\in\Sigma_v^1\cup\Sigma_p^1} K_e) \ot (\bigot_{w\in\Sigma_p^0\setminus\{v\}} Z_w^*)$, with the corresponding left actions of 
$K_{\halfedges{v}}$ and $H_p^*$.
With the notation of Subsection \ref{subsec:half-edges-and-sites}, let $e_p, e'_p \in \halfedges{v}$ be the half-edges at $v$ on the two sides of the site $p \in \vtxsites{v}$, with signs $\eps:=\vertexedge{v}{e_p}$ and $\eps':= \vertexedge{v}{e'_p}$.
The $K_{\halfedges{v}}$- and $H_p^*$-actions only overlap on the tensor factors $(K_e)_{e\in\Sigma_v^1\cap\Sigma_p^1}$ corresponding to the edges underlying the half-edges $e_p, e'_p \in \halfedges{v}$.
Due to our regularity assumption on the cell decomposition, the half-edges $e_p$ and $e'_p$ have distinct underlying edges.
Then the action of $K_\halfedges{v} = (K_{e_p}^{\eps} \ot K_{e'_p}^{\eps'}) \ot \bigot_{e\in\halfedges{v}\setminus\{e_p,e'_p\}} K_e^\vertexedge{v}{e}$ on $\bigot_{e\in\Sigma_v^1} K_e$, which is a tensor product of algebras, decomposes into a tensor product of the action of $K_{e_p}^{\eps} \ot K_{e'_p}^{\eps'}$ on $K_{e_p} \ot K_{e'_p}$ and the action of $\bigot_{e\in\halfedges{v}\setminus\{e_p,e'_p\}} K_e^\vertexedge{v}{e}$ on $\bigot_{e\in\Sigma_v^1\setminus\{e_p,e'_p\}} K_e$.
On the latter vector space, $H_p^*$ does not act non-trivially by our regularity assumption on the cell decomposition.
Hence, it remains to consider the interactions of the left actions of $K_{e_p}^{\eps} \ot K_{e'_p}^{\eps'}$ and $H_p^*$ on the vector space $K_{e_p} \ot K_{e'_p} \ot (\bigot_{e\in\Sigma_p^1\setminus\{e_p,e'_p\}} K_e) \ot (\bigot_{w\in\Sigma_p^0\setminus\{v\}} Z_w^*)$.
We abbreviate by $V := (\bigot_{e\in\Sigma_p^1\setminus\{e_p,e'_p\}} K_e) \ot (\bigot_{w\in\Sigma_p^0\setminus\{v\}} Z_w^*)$ the tensor factor on which only $H_p^*$ acts non-trivially.
Furthermore, without loss of generality, we write the left $H_p^*$-action on $V$ in terms of the Sweedler notation for the corresponding right $H_p$-coaction, $V\to V\ot H_p, v \mapsto v_{(0)} \ot v_{(1)}$:
$$ H_p^* \ot V \lto V,v \lmapsto f.v =: f(v_{(1)}) v_{(0)} . $$
Finally, it is left to analyze the interaction between the $H_p^*$-action
\begin{align*}
H_p^* \ot K_{e_p} \ot K_{e'_p} \ot V &\lto K_{e_p} \ot K_{e'_p} \ot V,	\\
f \ot x \ot x' \ot v &\lmapsto f_{(3)}.x \ot f_{(1)}.x' \ot f_{(2)}.v	\\
&\phantom{\lmapsto} = f\left( \anti{{x'}_{(\eps')}}{\eps'} v_{(1)} \anti{{x}_{(-\eps)}}{-\eps} \right) x_{(0)} \ot x'_{(0)} \ot v_{(0)},
\end{align*}
and the $(K_{e_p}^{\eps} \ot K_{e'_p}^{\eps'})$-action
\begin{align*}
(K_{e_p}^{\eps} \ot K_{e'_p}^{\eps'}) \ot K_{e_p} \ot K_{e'_p} \ot V &\lto K_{e_p} \ot K_{e'_p} \ot V,	\\
a \ot a' \ot x \ot x' \ot v &\lmapsto a.x \ot a'.x' \ot v	\\
&\phantom{\lmapsto} (a \cdot_\eps x) \ot (a' \cdot_{\eps'} x') \ot v,
\end{align*}
where $\cdot_\eps$ and $\cdot_{\eps'}$ denote the multiplication in $K_{e_p}^\eps$ and $K_{e'_p}^{\eps'}$, respectively,
that is \[ a \cdot_\eps x := \begin{cases}
a x, & \eps = +1, \\
x a, & \eps = -1.
\end{cases} \]
It remains to show that that these actions satisfy the straightening formula
\begin{equation*}
f(\anti{{a'}_{(\eps')}}{-\eps'} \cdot ? \cdot \anti{a_{(-\eps)}}{\eps}).(a_{(0)} \ot a'_{(0)}).(x \ot x' \ot v) = (a \ot a').f.(x \ot x' \ot v),
\end{equation*}
for all $f \in H_p^*$, $a\ot a' \in K_{e_p}^{\eps} \ot K_{e'_p}^{\eps'}$ and $x \ot x' \ot v \in K_{e_p} \ot K_{e'_p} \ot V$.
Indeed, the following calculation, which is analogous to the calculation in the proof of \cite[Theorem~1]{buerschaperEtAl} but more general and at the same time shorter, verifies this.
\begin{align*}
& f \Big( \anti{{a'}_{(\eps')}}{-\eps'} \cdot ? \cdot \anti{{a}_{(-\eps)}}{\eps} \Big) . (a_{(0)} \ot a'_{(0)}) . (x \ot x' \ot v) \\
&\phantom{xxxxxxxxx} = f \Big( \anti{{a'}_{(\eps')}}{-\eps'} \cdot ? \cdot \anti{{a}_{(-\eps)}}{\eps} \Big) . ((a_{(0)} \cdot_\eps x) \ot (a'_{(0)} \cdot_{\eps'} x') \ot v) \\
&\phantom{xxxxxxxxx}= f \Big( \anti{{a'}_{(2\eps')}}{-\eps'} \cdot \anti{(a'_{(0)} \cdot_{\eps'} x')_{(\eps')}}{\eps'} \cdot v_{(1)} \cdot \anti{(a_{(0)} \cdot_\eps x)_{(-\eps)}}{-\eps} \cdot \anti{{a}_{(-2\eps)}}{\eps} \Big) \\
&\phantom{xxxxxxxxx} \phantom{xxx} ((a_{(0)} \cdot_\eps x)_{(0)} \ot (a'_{(0)} \cdot_{\eps'} x')_{(0)} \ot v_{(0)}) \\
&\phantom{xxxxxxxxx}= f \Big( \anti{{a'}_{(2\eps')}}{-\eps'} \cdot \anti{{a'}_{(\eps')}}{\eps'} \cdot \anti{{x'}_{(\eps')}}{\eps'} \cdot v_{(1)} \cdot \anti{x_{(-\eps)}}{-\eps} \cdot \anti{a_{(-\eps)}}{-\eps} \cdot \anti{{a}_{(-2\eps)}}{\eps} \Big) \\
&\phantom{xxxxxxxxx} \phantom{xxx} ((a_{(0)} \cdot_\eps x_{(0)}) \ot (a'_{(0)} \cdot_{\eps'} x'_{(0)}) \ot v_{(0)}) \\
&\phantom{xxxxxxxxx}= f \Big( \anti{{x'}_{(\eps')}}{\eps'} \cdot v_{(1)} \cdot \anti{x_{(-\eps)}}{-\eps} \Big) ( (a \cdot_\eps x_{(0)}) \ot (a' \cdot_{\eps'} x'_{(0)}) \ot v_{(0)} ) \\
&\phantom{xxxxxxxxx}= (a \ot a') . \bigg( f \Big( \anti{{x'}_{(\eps')}}{\eps'} \cdot v_{(1)} \cdot \anti{x_{(-\eps)}}{-\eps} \Big) ( x_{(0)} \ot x'_{(0)} \ot v_{(0)} ) \bigg) \\
&\phantom{xxxxxxxxx}= (a \ot a') . \big( f . ( x \ot x' \ot v ) \big) .
\end{align*}
This proves that $H_p^*$ and $K_{e_p}^{\eps} \ot K_{e'_p}^{\eps'}$ together give a representation of the crossed product algebra $\vtxal{H_p^*}{(K_{e_p}^{\eps} \ot K_{e'_p}^{\eps'})}$, as claimed.
\end{itemize}
\end{proof}

\begin{remark}
Taking all sites $p \in \vtxsites{v}$ around a given vertex $v \in \Sigma^0$ together, we thus get, due to Theorem \ref{thm:straightening-formula}, on $\statespace$ a bimodule structure over the vertex algebra $C_v$.
It is remarkable that this makes the crossed product algebra structure on $C_v$ show up naturally -- analogous to the appearance of the algebra structure of the Drinfeld double in the commutation relation of the vertex and plaqette actions in the standard Kitaev model without defects.
\end{remark}

\subsection{Towards local projectors: Symmetric separability idempotents for bicomodule algebras} \label{subsec:symm-sep-idem}

Before we proceed to use the bimodule structures on the state space $\statespace$ defined in Subsection \ref{subsec:local-reps-on-state-space} to define commuting local projectors on the vector space $\statespace$, we need to invoke another algebraic ingredient.

The standard Kitaev construction for a semisimple Hopf algebra $H$ makes use of the Haar integrals of $H$ and of $H^*$, in order to define commuting local projectors on the state space via the actions of $H$ and $H^*$.
The \emph{Haar integral} of a semisimple Hopf algebra $H$ over $\kk$ is the unique element $\ell \in H$ satisfying $x \ell = \eps(x) \ell = \ell x$ for all $x \in H$ and $\eps(\ell)=1$.
This means that $\ell$ is the central idempotent which projects to the $H$-invariants: for any $H$-module $M$, we have $\ell.M = M^H := \{ m \in M \mid h.m = \eps(h) m \quad\forall h\in H \}$.
Furthermore, $\ell \in H$ is cocommutative, i.e.\ $\ell_{(1)} \ot \ell_{(2)} = \ell_{(2)} \ot \ell_{(1)}$ in Sweedler notation.
The idempotence, centrality and cocommutativity of the Haar integral are crucial in showing that the Haar integral gives rise to commuting local projectors in the standard Kitaev construction \cite{buerschaperEtAl}.

In our setting, instead of a semisimple Hopf algebra acting on the state space, we have, for each vertex $v\in\Sigma^0$, a bimodule structure on the state space over a semisimple (bi-)comodule algebra $K_{\halfedges{v}}$.
Hence, we need a notion replacing the Haar integral, that works in this setting.
Our main insight is that the suitable generalization of the Haar integral to our setting is the unique symmetric separability idempotent, which exists for any semisimple algebra over an algebraically closed field $\kk$ with characteristic zero.
\begin{definition} \label{def:symm-sep-idem}
Let $A$ be an algebra over a field $\kk$.
A \emph{symmetric separability idempotent} for $A$ is an element $p \in A \ot A$, which we write as $p = p^1 \ot p^2 \in A \ot A$ omitting the summation symbol, satisfying
\begin{align}
(x \cdot p^1) \ot p^2 &= p^1 \ot (p^2 \cdot x)	\quad\forall x\in A,	\label{eq:invariance-of-sep-idem}	\\
p^1\cdot p^2 &= 1,	\label{eq:normalization-of-sep-idem} \\
p^1 \ot p^2 &= p^2 \ot p^1, \qquad\qquad\text{(symmetry)}	\label{eq:symmetry-of-sep-idem}
\end{align}
where on both sides of equation \eqref{eq:invariance-of-sep-idem} and in equation \eqref{eq:normalization-of-sep-idem} we are using the multiplication in $A$.

The properties \eqref{eq:invariance-of-sep-idem} and \eqref{eq:normalization-of-sep-idem} immediately imply that $p^1 \ot p^2$ is an idempotent when seen as an element of the enveloping algebra $A \ot A^\opp$.
\end{definition}
\begin{remarks} \ \begin{enumerate}[ref={\theremark.(\arabic*)}]
\item
The structure of a separability idempotent, i.e. an element $p^1\ot p^2 \in A \ot A$ satisfying \eqref{eq:invariance-of-sep-idem} and \eqref{eq:normalization-of-sep-idem}, is equivalent to an $A$-bimodule map $s : A \lto A \ot A$ that is a section of the multiplication $m : A \ot A \lto A$, by defining $s(x) := p^1 \ot p^2 x$ for all $x\in A$.
An algebra endowed with such a structure is called \emph{separable} and, in general, such a separability structure might not exist or be unique.
A symmetric separability structure, however, is always unique -- see the end of the proof of Proposition \ref{prop:symm-sep-idem-unique}.
\item
Representation-theoretically, a separability idempotent $p^1 \ot p^2 \in A \ot A^\opp$ plays the role of projecting to the subspace of invariants for any $A$-bimodule $M$.
Indeed, due to property \eqref{eq:invariance-of-sep-idem}, one has
\[ p^1.M.p^2 = M^A := \{ m\in M \mid a.m = m.a \quad\forall a\in A \} \subseteq M. \]
This is in analogy to the Haar integral $\ell \in H$ of a semisimple Hopf algebra $H$ which projects to the invariants $\ell.M = M^H := \{ m \in M \mid h.m = \eps(h) m \quad\forall h\in H \}$ of any left $H$-module $M$.
\end{enumerate}
\end{remarks}

Just as every finite-dimensional semisimple Hopf algebra over a field $\kk$ has a unique Haar integral, for every finite-dimensional semisimple $\kk$-algebra there exists a unique symmetric separability idempotent:
\begin{proposition}[\cite{aguiar}]
	\label{prop:symm-sep-idem-unique}
Let $A$ be a finite-dimensional semisimple algebra over a field $\kk$ which is algebraically closed and of characteristic zero.
Then there exists a unique symmetric separability idempotent $p^1 \ot p^2 \in A \ot A^\opp$ for $A$.
\end{proposition}
\begin{proof}
For a more detailed proof, see \cite[Thm.~3.1,~Cor.~3.1.1]{aguiar}.
Here we recall the main idea that the unique symmetric separability idempotent can be described in terms of the trace form on $A$, because we will use this description in Proposition \ref{prop:symm-sep-idem-coinvariant}.

Due to semisimplicity, the following symmetric bilinear pairing on $A$ is non-degenerate:
\begin{align*}
T : A \ot A &\lto \kk, \\
a \ot b &\lmapsto t(a \cdot b) := \tr_A (L_{a\cdot b}),
\end{align*}
defined in terms of the trace form
where $L_?$ denotes the left multiplication of $A$.
In fact, this non-degenerate bilinear pairing turns $A$ into a symmetric special Frobenius algebra.
Consider the isomorphism $\#_T : A \simto A^*, a \mapsto t(a \cdot -),$ induced by this non-degenerate bilinear pairing.
This is an isomorphism of $A$-bimodules.
It induces an isomorphism $A \ot A \simto A^* \ot A \cong \End{\kk}{A}$.
Consider the pre-image $p \in A \ot A$ of the identity $\id{A}$ under this isomorphism.
As usual, we write an element $p \in A\ot A$ as $p = p^1 \ot p^2$, omitting the summation symbol.
In fact, if we choose a basis $(p^1_i)_i$ for $A$ and let $(p^2_i)_i$ be its dual basis of $A$ with respect to the non-degenerate pairing $T$, then $p^1 \ot p^2$ is the sum $\sum_i p^1_i \ot p^2_i$.
With this definition of $p^1 \ot p^2 \in A \ot A$ it is straightforward to verify the defining properties \eqref{eq:invariance-of-sep-idem}, \eqref{eq:normalization-of-sep-idem} and \eqref{eq:symmetry-of-sep-idem} of a symmetric separability idempotent.

To prove that the symmetric separability idempotent is unique, let $p^1 \ot p^2, q^1 \ot q^2 \in A \ot A^\opp$ be any two symmetric separability idempotents for $A$.
Then they are equal by the following computation:
\begin{align*}
	p^1 \ot p^2
&\stackrel{\eqref{eq:normalization-of-sep-idem}}{=}	q^1 q^2 p^1 \ot p^2
\stackrel{\eqref{eq:invariance-of-sep-idem}}{=}		q^1 p^1 \ot p^2 q^2
\stackrel{\eqref{eq:symmetry-of-sep-idem}}{=}		q^2 p^1 \ot p^2 q^1		\\
&\stackrel{\eqref{eq:invariance-of-sep-idem}}{=}	q^2 \ot p^2 p^1 q^1
\stackrel{\eqref{eq:symmetry-of-sep-idem}}{=}		q^2 \ot p^1 p^2 q^1
\stackrel{\eqref{eq:normalization-of-sep-idem}}{=}	q^2 \ot q^1
\stackrel{\eqref{eq:symmetry-of-sep-idem}}{=}		q^1 \ot q^2,
\end{align*}
using the defining properties \eqref{eq:invariance-of-sep-idem}, \eqref{eq:normalization-of-sep-idem} and \eqref{eq:symmetry-of-sep-idem}.
\end{proof}

\begin{example} \label{ex:haar-int-as-symm-sep-idem}
Let $H$ be a finite-dimensional semisimple Hopf algebra over $\kk$ with Haar integral $\ell \in H$.
Then the symmetric separability idempotent for $H$ is $\ell_{(1)} \ot S(\ell_{(2)}) \in H \ot H^\opp$.

Indeed, the invariance property of the Haar integral, $x \ell = \eps(x) \ell$ for all $x\in H$, implies the corresponding invariance property \eqref{eq:invariance-of-sep-idem} of $\ell_{(1)} \ot S(\ell_{(2)})$.
The normalization $\eps(\ell) = 1$ of the Haar integral implies the corresponding normalization property \ref{eq:normalization-of-sep-idem} for the separability idempotent.
Finally, using that the Haar integral is two-sided, which implies $S(\ell) = \ell$, it can be shown that $\ell_{(1)} \ot S(\ell_{(2)})$ is symmetric.

Hence we see that, in the sense of this example, the symmetric separability idempotent of a semisimple algebra generalizes the Haar integral of a semisimple Hopf algebra.
\end{example}

In our construction of a Kitaev model, however, we are not only dealing with semisimple algebras, but semisimple algebras together with a compatible bicomodule structure.
On the other hand, the Haar integral $\ell \in H$ has the property of being cocommutative, $\ell_{(1)} \ot \ell_{(2)} = \ell_{(2)} \ot \ell_{(1)}$, which is crucial in showing that it gives rise to commuting projectors in \cite{buerschaperEtAl}
and we have not exhibited an analogous property of the symmetric separability idempotent.
In the following proposition we prove such a property, which holds for the symmetric separability idempotent of a semisimple (bi-)comodule algebra and which generalizes the cocommutativity of the Haar integral, see Example \ref{ex:coinv-of-reg-symm-sep-idem}.

\begin{proposition}\label{prop:symm-sep-idem-coinvariant}
Let $H$ be a semisimple Hopf algebra over $\kk$ and let $K$ be a semisimple right $H$-comodule algebra with symmetric separability idempotent $p^1 \ot p^2 \in K \ot K^\opp$.
Consider the right $H$-coaction on the tensor product $K \ot K^\opp$: 
\begin{align*}
K \ot K^\opp &\lto K \ot K^\opp \ot H,	\\
k \ot k' &\lmapsto k_{(0)} \ot k'_{(0)} \ot k_{(1)} k'_{(1)}	.
\end{align*}
Then $p^1 \ot p^2 \in K \ot K^\opp$ is an $H$-coinvariant element of $K \ot K^\opp$, i.e.\ $p^1_{(0)} \ot p^2_{(0)} \ot p^1_{(1)} p^2_{(1)} = p^1 \ot p^2 \ot 1_H \in K\ot K^\opp \ot H$, and this is equivalent to
\begin{equation}\label{eq:symm-sep-idem-cyclic} p^1_{(0)} \ot p^1_{(1)} \ot p^2 = p^1 \ot S(p^2_{(1)}) \ot p^2_{(0)} \in K \ot H \ot K^\opp	.  \end{equation}
Analogously, if $K$ is a left $H$-comodule algebra, then
\begin{equation}\label{eq:symm-sep-idem-cyclic_left-version}
p^1_{(0)} \ot p^1_{(-1)} \ot p^2 = p^1 \ot S(p^2_{(-1)}) \ot p^2_{(0)} \in K \ot H \ot K^\opp	.  \end{equation}
\end{proposition}
\begin{proof}
Without loss of generality we only show the case where $K$ is a right $H$-comodule algebra.
Recall from the proof of Proposition \ref{prop:symm-sep-idem-unique} that the symmetric separability idempotent $p^1 \ot p^2 \in K \ot K^\opp$ for $K$ can be characterized in terms of the multiplication and the trace form $t : K \lto \kk$ on $K$, namely by $t(p^1\cdot x) p^2 = x \ \forall x\in K$, as explained in the proof of Proposition \ref{prop:symm-sep-idem-unique}.
Another way of phrasing this is that the map $K^* \lto K$ defined by $f \lmapsto f(p^1) p^2$ is the inverse of the isomorphism $K \lto K^*, k \lmapsto t(?\cdot k)$ induced by the non-degenerate pairing $t \circ \mu$, where $\mu : K \ot K \lto K$ is the multiplication on $K$.

The crucial step for the present proof is the observation that the multiplication and the trace form on $K$ are morphisms of $H$-comodules if $K$ is an $H$-comodule algebra.
For the multiplication this means that $x_{(0)} y_{(0)} \ot x_{(1)} y_{(1)} = (xy)_{(0)} \ot (xy)_{(1)} \ \forall x,y \in K$, which holds by definition of a comodule algebra, see Definition \ref{def:bicomodule-algebra}.
As for the $H$-colinearity of the trace form, note that $t = \ev_K \circ (\mu \ot \id{K^*}) \circ (\id{K} \ot \coev_K)$, where $\mu: K\ot K \to K$ denotes the multiplication, and $\coev_K : \kk \lto K \ot K^*$ and $\ev_K : K \ot K^* \lto \kk$ are the standard coevaluation and evaluation morphisms for vector spaces.
Due the involutivity of the antipode $S$ of $H$, both $\ev_K$ and $\coev_K$ are morphisms of right $H$-comodules for the $H$-comodule structure on the dual $K^*$ given by $K^* \lto K^* \ot H, \varphi \lmapsto \varphi_{(0)} \ot \varphi_{(1)}$, where $\varphi_{(0)}(x) \varphi_{(1)} := \varphi(x_{(0)}) S(x_{(1)})$ for all $x \in K$.
(We are here implicitly using the canonical trivial pivotal structure on the tensor category of right $H$-comodules, which exists due to the involutivity of the antipode of $H$.)
Since therefore the trace form $t$ is composed only of morphisms of right $H$-comodules, it is itself a morphism of right $H$-comodules, i.e.\
\begin{equation}\label{eq:trace-form-colinear}	t(k_{(0)}) k_{(1)} = t(k) 1_H \quad \forall k\in K . \end{equation}

As a consequence, the isomorphism $K \lto K^*, k \lmapsto t(?\cdot k)$ induced by the pairing $t\circ \mu$ is an isomorphism of $H$-comodules.
Indeed, for all $x \in K$ one has $t(x k_{(0)}) k_{(1)} = t(x_{(0)} k_{(0)}) S(x_{(2)}) x_{(1)} k_{(1)} \stackrel{\eqref{eq:trace-form-colinear}}{=} t(x_{(0)} k) S(x_{(1)}) \eqbydef (t(?\cdot k))_{(0)}(x) (t(?\cdot k))_{(1)}$. 

This immediately implies that the inverse map, $K^* \lto K, \varphi \lmapsto \varphi(p^1)p^2$, must also be a morphism of $H$-comodules, which spelled out means that $\varphi(p^1_{(0)})p^2 \ot S(p^1_{(1)}) \eqbydef \varphi_{(0)}(p^1)p^2\ot \varphi_{(1)} = \varphi(p^1) p^2_{(0)} \ot p^2_{(1)}$ for all $\varphi \in K^*$.
This implies the equation \eqref{eq:symm-sep-idem-cyclic} of the claim.
To show that this is equivalent to $p^1 \ot p^2 \in K\ot K^\opp$ being $H$-coinvariant, we compute $$p^1_{(0)} \ot p^2_{(0)} \ot p^1_{(1)} p^2_{(1)} \stackrel{\eqref{eq:symm-sep-idem-cyclic}}{=} p^1 \ot p^2_{(0)} \ot S(p^2_{(1)}) p^2_{(2)} = p^1 \ot p^2 \ot 1_H . $$
\end{proof}
\begin{example} \label{ex:coinv-of-reg-symm-sep-idem}
Let $H$ be a semisimple Hopf algebra and consider it as the regular $H$-bicomodule algebra, as in Example \ref{ex:reg-bicomodule-alg}.
Recall that for $H$ the symmetric separability idempotent is $p^1 \ot p^2 = \ell_{(1)} \ot S(\ell_{(2)}) \in H \ot H$.
Let us spell out Proposition \ref{prop:symm-sep-idem-coinvariant} for the left and right $H$-comodule structures on the regular bicomodule algebra $H$.
Equation \eqref{eq:symm-sep-idem-cyclic} boils down to the equation $(\ell_{(1)})_{(1)} \ot (\ell_{(1)})_{(2)} \ot S(\ell_{(3)}) = \ell_{(1)} \ot S(S(\ell_{(2)})_{(2)}) \ot S(\ell_{(2)})_{(1)}$.
But due to $S^2 = \id{H}$ both sides of the equation are equal to $\ell_{(1)} \ot \ell_{(2)} \ot S(\ell_{(3)})$.
On the other hand, equation \eqref{eq:symm-sep-idem-cyclic_left-version} boils down to the equation $(\ell_{(1)})_{(2)} \ot (\ell_{(1)})_{(1)} \ot S(\ell_{(3)}) = \ell_{(1)} \ot S(S(\ell_{(2)})_{(1)}) \ot S(\ell_{(2)})_{(2)}$, which in turn due to $S^2 = \id{H}$ simplifies to $\ell_{(2)} \ot\ell_{(1)} \ot S(\ell_{(3)}) = \ell_{(1)} \ot \ell_{(3)} \ot S(\ell_{(2)})$.
This is equivalent to the cocommutativity property $\ell_{(1)} \ot \ell_{(2)} = \ell_{(2)} \ot \ell_{(1)}$.

Hence we have shown that the coinvariance property of the symmetric separability idempotent for a bicomodule algebra, proven in Proposition \ref{prop:symm-sep-idem-coinvariant}, is the appropriate analogue of the cocommutativity of the Haar integral.
In the proof of Lemma \ref{lem:symm-sep-idempotents-commute-in-drinfeld-double} we will use it in a crucial way, on the way towards proving in Theorem \ref{thm:vertex-and-plaquette-operators-commute} that symmetric separability idempotents allow for defining commuting projectors.
\end{example}

\begin{lemma} \label{lem:symm-sep-idempotents-commute-in-drinfeld-double}
Let $H$ be a semisimple Hopf algebra over $\kk$ and let $K$ be a semisimple left $H$-comodule algebra and $A$ a semisimple left $H$-module algebra.
Let $p^1 \ot p^2 \in K \ot K^\opp$ and $\pi^1 \ot \pi^2 \in A \ot A^\opp$ be the symmetric separability idempotents for $K$ and $A$, respectively.

Then $(1_A\ot p^1) \ot (1_A\ot p^2)$ and $(\pi^1 \ot 1_K) \ot (\pi^2 \ot 1_K)$ commute in the algebra $(\vtxal{A}{K}) \ot (\vtxal{A}{K})^\opp$, where $\vtxal{A}{K}$ is the crossed product algebra defined in Definition \ref{def:crossed-product-algebra}.
\end{lemma}
\begin{proof}
Due to the co-invariance of the symmetric separability idempotent of a semisimple comodule algebra over $\kk$, proven in Proposition \ref{prop:symm-sep-idem-coinvariant}, we have $$p^1_{(-1)} \ot p^1_{(0)} \ot p^2 \stackrel{\eqref{eq:symm-sep-idem-cyclic_left-version}}{=} S(p^2_{(-1)}) \ot p^1 \ot p^2_{(0)}$$ and $$(h.\pi^1) \ot \pi^2 = \pi^1 \ot (S(h).\pi^2)$$ for all $h \in H$, where the latter can be derived from equation \eqref{eq:symm-sep-idem-cyclic} by regarding $A$ as a right $H^*$-comodule algebra, which is equivalent to a left $H$-module algebra \cite{montgomery}.
By definition of the multiplication in $(\vtxal{A}{K}) \ot (\vtxal{A}{K})^\opp$ we have:
\[ (1_A\ot p^1) \ot (1_A\ot p^2) \cdot (\pi^1 \ot 1_K) \ot (\pi^2 \ot 1_K) = (p^1_{(-1)}.\pi^1 \ot p^1_{(0)}) \ot (\pi^2 \ot p^2) \]
and
\[ (\pi^1 \ot 1_K) \ot (\pi^2 \ot 1_K) \cdot (1_A\ot p^1) \ot (1_A\ot p^2) = (\pi^1 \ot p^1) \ot (p^2_{(-1)} . \pi^2 \ot p^2_{(0)}) \]
But the right-hand sides of these equations are equal by the following computation:
\begin{align*}
(p^1_{(-1)}.\pi^1 \ot p^1_{(0)}) \ot (\pi^2 \ot p^2) &= (S(p^2_{(-1)}).\pi^1 \ot p^1) \ot (\pi^2 \ot p^2_{(0)}) \\
&= (\pi^1 \ot p^1) \ot (S^2(p^2_{(-1)}).\pi^2 \ot p^2_{(0)}) \\
&= (\pi^1 \ot p^1) \ot (p^2_{(-1)}.\pi^2 \ot p^2_{(0)}) .
\end{align*}
\end{proof}

\subsection{Local commuting projector Hamiltonian from vertex and plaquette operators}
\label{subsec:vertex-plaquette-operators}

In this subsection we define on the vector space $\statespace$ assigned to a surface $\Sigma$ with a labelled cell decomposition a set of commuting local projectors and finally, in the spirit of Kitaev lattice models, a Hamiltonian on $\statespace$ as the sum of commuting projectors.

Recall that in Subsection \ref{subsec:local-reps-on-state-space} we have defined on $\statespace$ a $K_{\halfedges{v}}$-bimodule structure $\tild{A}_v$ for each vertex $v\in\Sigma^0$ and a $H_p^*$-bimodule structure $\tild{B}_{(p,v)}$ for each site $(p,v)$, $p\in\Sigma^2$, $v\in\plaqsites{p}$.
A $K_{\halfedges{v}}$-bimodule structure is equivalent to a left $(K_{\halfedges{v}} \ot K_{\halfedges{v}}^\opp)$-action on $\statespace$, so that specifying an element of the so-called enveloping algebra $(K_{\halfedges{v}} \ot K_{\halfedges{v}}^\opp)$ determines an endomorphism of $\statespace$.
By assumption, all bicomodule algebras $K_e$ labelling the cell decomposition of $\Sigma$ are semisimple and, hence, the tensor product $K_\halfedges{v}$ is semisimple and possesses a unique symmetric separability idempotent $p_v^1 \ot p_v^2 \in (K_{\halfedges{v}} \ot K_{\halfedges{v}}^\opp)$ according to Proposition \ref{prop:symm-sep-idem-unique}.

\begin{definition} \label{def:vertex-operator}
Let $v \in \Sigma^0$.
The \emph{vertex operator} for the vertex $v$ is the idempotent endomorphism of the state space $\statespace$
\begin{equation*}
A_v := \tild{A}_v (p_v^1\ot p_v^2) : \statespace \lto \statespace
\end{equation*}
given by acting with the unique symmetric separability idempotent $$p_v^1 \ot p_v^2 \in K_{\halfedges{v}} \ot K_{\halfedges{v}}^\opp$$ via the $K_{\halfedges{v}}$-bimodule structure $\tild{A}_v$, defined in Definition \ref{def:vertex-bimodule}.
\end{definition}
This operator is \emph{local} in the sense that it acts as the identity on all tensor factors in $\statespace = (\ot_{e\in\Sigma^1} K_e^*) \ot (\ot_{w\in\Sigma^0} Z_w)$ except for those associated to the vertex $v \in \Sigma^0$ and to the edges $e\in\Sigma_v^1$ incident to $v$.
Since the symmetric separability idempotent of a semisimple bicomodule algebra generalizes the Haar integral of a semisimple Hopf algebra, as explained in Subsection \ref{subsec:symm-sep-idem}, we see that the vertex operator defined here provides a suitable analogon to the vertex operators in the ordinary Kitaev model for a semisimple Hopf algebra.

\medskip

Next we want to define a projector on $\statespace$ for each plaquette $p \in \Sigma^2$ in analogy to the plaquette operators of the ordinary Kitaev model for a semisimple Hopf algebra $H$, which are defined by acting with the Haar integral of the dual Hopf algebra $H^*$.
In our construction, we have defined in Definition \ref{def:plaquette-bimodule} an $H_p^*$-bimodule structure $\tild{B}_{(p,v)}$ on $\statespace$ for every plaquette $p \in \Sigma^2$ with incident site $v\in\plaqsites{p}$ and we can again use this to define a projector $\tild{B}_{(p,v)}({\lambda_p}_{(1)} \ot S({\lambda_p}_{(2)}))$ on $\statespace$ by acting with the symmetric separability idempotent of the semisimple algebra $H_p^*$, which is ${\lambda_p}_{(1)} \ot S({\lambda_p}_{(2)}) \in H_p^* \ot (H_p^*)^\opp$, see Example \ref{ex:haar-int-as-symm-sep-idem}.
However note that, as opposed to the vertex operator here it is actually not necessary to invoke the concept of the symmetric separability idempotent, since $H_p^*$ is a Hopf algebra just as in the ordinary Kitaev model, and its symmetric separability idempotent is given by the Haar integral.

When considering the projector $\tild{B}_{(p,v)}({\lambda_p}_{(1)} \ot S({\lambda_p}_{(2)}))$ on $\statespace$, it seems that a priori it depends not only on the plaquette $p$ but also on the site $v \in \plaqsites{p}$ that we had to choose in Definition \ref{def:plaquette-bimodule} in order to define the bimodule structure $\tild{B}_{(p,v)}$.
Just like the plaquette operators in the ordinary Kitaev model, we will show that due to the properties of the Haar integral the projector only depends on the plaquette $p$: 

\begin{lemma}
Let $p \in \Sigma^2$.
If $\lambda_p \in H_p^*$ is the Haar integral of $H_p^*$,
then the endomorphism $$\tild{B}_{(p,v)}({\lambda_p}_{(1)} \ot S({\lambda_p}_{(2)})) : \statespace \lto \statespace$$ does not depend on the choice of the site $v \in \plaqsites{p}$.
\end{lemma}
\begin{proof}
The endomorphism $\tild{B}_{(p,v)}({\lambda_p}_{(1)} \ot S({\lambda_p}_{(2)}))$
is equal to the endomorphism of $\statespace$ obtained by acting with the Haar integral $\lambda$ via the left $H_p^*$-action ${B}'_{(p,v)}$ on $\statespace$ that is the pullback of the left $(H_p^* \ot (H_p^*)^\opp)$-action $\tild{B}_{(p,v)}$ along the algebra map $(\id{H_p^*} \ot S) \circ \Delta : H_p^* \lto H_p^* \ot (H_p^*)^\opp$.
Next we observe that the action ${B}'_{(p,v)}$ is independent of $v$ for any cocommutative element $\lambda$ of the Hopf algebra $H_p^*$.
Indeed, looking carefully at Definition \ref{def:plaquette-bimodule}, we extract from it that ${B}'_{(p,v)}(\lambda)$ acts with the multiple coproduct of $\lambda$ on the degrees of freedom of $\statespace$ in the boundary of the plaquette $p$ in a cyclic order starting at the vertex $v$.
Therefore, for a different vertex $v' \in \plaqsites{p}$, the endomorphism ${B}'_{(p,v')}(\lambda)$ will only differ by a cyclic shift in the multiple coproduct of $\lambda$.
But since $\lambda$ is cocommutative, any multiple coproduct of it is invariant under such cyclic shifts of its tensor factors.
\end{proof}
Thus we have shown that the following is well-defined.
\begin{definition} \label{def:plaquette-operator}
Let $p\in\Sigma^2$.
The \emph{plaquette operator} for the plaquette $p$ is the idempotent endomorphism of the state space $\statespace$
\begin{equation*}
B_p := \tild{B}_{(p,v)}({\lambda_p}_{(1)} \ot S({\lambda_p}_{(2)})) : \statespace \lto \statespace	
\end{equation*}
given by acting via the $H_p^* \ot (H_p^*)^\opp$-action $\tild{B}_{(p,v)}$ introduced in Definition \ref{def:plaquette-bimodule} with the unique symmetric separability idempotent ${\lambda_p}_{(1)} \ot S({\lambda_p}_{(2)}) \in H_p^* \ot (H_p^*)^\opp$ for $H_p^*$.
Here $\lambda_p \in H_p^*$ is the Haar integral for $H_p^*$.
\end{definition}

This operator is \emph{local} in the sense that it acts as the identity on all tensor factors in $\statespace = (\ot_{e\in\Sigma^1} K_e^*) \ot (\ot_{v\in\Sigma^0} Z_v)$ except for those associated to the edges $e\in\Sigma_p^1$ and the vertices $v \in \Sigma_p^0$ incident to the plaquette $p$.

We have thus defined a family of projectors $(A_v)_{v\in\Sigma^0}$ and $(B_p)_{p\in \Sigma^2}$ on the vector space $\statespace$.
We now finally reach our main result that they all commute with each other.

\begin{theorem} \label{thm:vertex-and-plaquette-operators-commute}
Let $\Sigma$ be an oriented compact surface with a regular cell decomposition labeled by semisimple Hopf algebras, semisimple bicomodule algebras and representations of the vertex algebras, and
let $\statespace$ be the associated vector space defined in Definition \ref{def:statespace} with vertex and plaquette operators $\{ (A_v)_{v\in\Sigma^0}, (B_p)_{p\in\Sigma^2} \}$ defined in Definitions \ref{def:vertex-operator} and \ref{def:plaquette-operator}.

Then any pair of vertex or plaquette operators commutes.
\end{theorem}
\begin{proof}
Due to Theorem \ref{thm:straightening-formula}, the only non-trivial commutation relations between a $K_\halfedges{v}$-action and an $H_p^*$-action on $\statespace$ may occur when $v$ and $p$ are incident to each other.
In that case, the $K_\halfedges{v}$-bimodule structure $\tild{A}_v$ and the $H_p^*$-bimodule structure $\tild{B}_{(p,v)}$ together form a bimodule structure over the crossed product algebra $\vtxal{H_p^*}{K_\halfedges{v}}$.
However, due to Lemma \ref{lem:symm-sep-idempotents-commute-in-drinfeld-double} the symmetric separability idempotents for $K_\halfedges{v}$ and $H_p^*$ commute in $(\vtxal{H_p^*}{K_\halfedges{v}}) \ot (\vtxal{H_p^*}{K_\halfedges{v}})^\opp$ and, hence, the vertex operator $A_v$ and the plaquette operator $B_p$ commute with each other.
\end{proof}

This is completely analogous to the standard Kitaev model without defects:
We have a family of commuting projectors on the state space.
Since any family of commuting projectors is simultaneously diagonalizable, this allows for the definition of an exactly solvable Hamiltonian as the sum of commuting projectors.
We thus conclude our construction of the Kitaev lattice model with defects as follows:

\begin{definition}
The \emph{Hamiltonian} on the state space $\statespace$  assigned to an oriented surface $\Sigma$ with labeled cell decomposition as above is the diagonalizable endomorphism
\begin{equation*}
h := \sum_{v \in \Sigma^0} (\id{\statespace} - A_v) + \sum_{p \in \Sigma^2} (\id{\statespace} - B_p) : \statespace \lto \statespace .
\end{equation*} 
The associated \emph{ground-state space} is its kernel,
\begin{equation*} \label{eq:ground-state-space}
\statespace_0 := \ker h ,
\end{equation*}
i.e.\! the simultaneous $1$-eigenspace for all the projectors $\{ (A_v)_{v\in\Sigma^0}, (B_p)_{p\in\Sigma^2} \}$.

Such a Hamiltonian is also called \emph{frustration-free}, as its lowest eigenvalue is not lower than any eigenvalue of its summands.
\end{definition}

\begin{remark}
The ground-state space $\statespace_0$ is isomorphic to the vector space that is category-theoretically realized by the modular functor constructed in \cite{fss} for the defect surface $\Sigma$ labeled by the corresponding representation categories of the Hopf algebras and bicomodule algebras.
We leave the detailed proof of this statement for a future update of this paper.

As a consequence, the ground-state space $\statespace_0$ is invariant under fusion of defects and independent of the transparently labeled part of the cell decomposition.
Moreover, due to the results of \cite{fss}, there will be a mapping class group action on $\statespace_0$ that can be explicitly computed.
This allows to define quantum gates on the ground-state space in terms of the mapping class group action, as has been proposed before, and to address questions of universality of such gates.
We have thus constructed an explicit Hamiltonian model which offers the possibility for quantum computation, realizing a general framework for theories of the type discussed e.g.\! in \cite{barkeshli}.
\end{remark}
A detailed investigation of the above and related questions remain for future work.

\pagebreak

\appendix

\section{A category-theoretic motivation for the vertex algebras} \label{app:gluing-category}

The construction in this paper takes as its input a compact oriented surface $\Sigma$, whose $2$-cells are labelled by
Hopf algebras and whose $1$-cells are labelled by
bicomodule algebras.
Furthermore, we have introduced in Definition \ref{def:vertex-algebra}, for every vertex $v \in \Sigma^0$, an algebra $C_v$, which we call vertex algebra.
The category of possible labels for a vertex $v \in \Sigma^0$ of the cell decomposition is the category of modules over the relevant vertex algebra $C_v$, see Definition \ref{def:rep-category-of-labels-at-vertex}.

On the other hand, in three-dimensional topological field theories and modular functors defined on surfaces with defects such as in \cite{fss, kitaevKong}, the strata are labelled by category-theoretic data: $2$-cells by finite tensor categories and $1$-cells by finite bimodule categories, which in our setting arise as the representation categories of the Hopf algebras and bicomodule algebras that we use as labels for our construction.

Furthermore, in \cite{fss}, a category is assigned to any boundary circle of a surface with defects, which is equivalent to a Drinfeld center in the absence of defects.
Such a boundary circle can be intersected by defect lines labelled by bimodule categories, leading to marked points on the circle.
In our construction this situation corresponds to a vertex $v \in \Sigma^0$ at which a number of edges labelled by bicomodule algebras meet.
We can regard such a vertex as a boundary circle $\mathbb{L}_v$, cut into the surface $\Sigma$, at which defect lines end which are labelled by the representation categories of the corresponding bicomodule algebras.

The main result of this section, Theorem \ref{thm:gluing-category-as-representation-category}, is that the category assigned to such a decorated circle with marked points $\mathbb{L}_v$ according to the prescription of \cite{fss}, defined in Definition \ref{def:gluing-category}, is canonically isomorphic to the category of labels that we have defined in Definition \ref{def:rep-category-of-labels-at-vertex} for such a vertex $v \in \Sigma^0$ in a labeled cell decomposition.		\\

\noindent First we must explain the category that is assigned to a boundary circle of a defect surface in the  construction of \cite{fss}.
For the category-theoretic background, see also \cite{egno}.
We adapt the notions and notation to our setting, since it slightly differs from the one in \cite{fss}.
Here, the tensor categories we consider are pivotal and the underlying defect surface is oriented, whereas in the reference no pivotal structures are used and instead the surfaces are framed.	\\

\noindent For a tensor category $\A$ and a sign $\eps \in \{ +1, -1 \}$, write \[
\A^\eps := \begin{cases}
\A, & \text{ if }\eps=+1,	\\
\overline{\A}, & \text{ if }\eps=-1,	
\end{cases}						\]
where $\overline{\A} := \A^{\opp,\mopp}$ is the tensor category whose underlying linear category is the opposite category of $\A$ and whose tensor product is also opposite to the one of $\A$, i.e. $\overline{a} \ot \overline{b} := \overline{b \ot a}$ for $a, b \in \A$ , where for any object $a \in \A$ we denote its corresponding object in the opposite category $\overline{\A}$ by $\overline{a}$, and likewise for morphisms.
If $\A = \lmod{H}$ for a Hopf algebra $H$, then $\overline{\A} \cong \lmod{\overline{H}}$ canonically as tensor categories, where $\overline{H} := H^{\opcop}$ is the Hopf algebra that has the opposite multiplication as well as the opposite co-multiplication with respect to $H$.
For $X \in \lmod{H}$, the corresponding object $\overline{X}$ in $\lmod{\overline{H}}$ is given by the vector space dual $\Hom{\kk}{X,\kk}$ of $X$ with the natural induced $\overline{H}$-action.
For $\eps \in \{+1,-1\}$, we also write $H^\eps := \overline{H}$ if $\eps = -1$, and $H^\eps := H$ if $\eps=+1$.

The right duality functor induces a monoidal equivalence, $\A \lto \overline{\A}, x \lmapsto \overline{x^\vee}$.
For $\A = \lmod{H}$ for a Hopf algebra $H$, this equivalence takes an $H$-module $X$ and turns it into an $\overline{H}$-module by pulling back the $H$-action along the antipode $S : \overline{H} \lto H$.
Note that instead of the right dual functor one can also take any other odd-fold right or left dual.
For our purposes this choice does not matter, since the tensor categories which we will consider are pivotal, where all these odd-fold duals are canonically identified.
Indeed, for a semisimple Hopf algebra $H$, the antipode is involutive, so that all odd powers of the antipode are the same.
(This is in contrast to \cite{fss} where no pivotal structures on the tensor categories are used, but instead $2$-framings on the underlying surfaces are used to determine which multiple of the duality functor to use in a given moment in the construction.)

If $\A_1$ and $\A_2$ are two tensor categories and $\M$ is an $\A_1$-$\A_2$-bimodule category, then the opposite linear category $\overline{\M} := \M^\opp$ canonically becomes an $\overline{\A_2}$-$\overline{\A_1}$-bimodule category by defining $\overline{a_2} \lact \overline{m} \ract \overline{a_1} := \overline{a_1 \lact m \ract a_2}$ for $a_1 \in \A_1$, $m \in \M$, $a_2 \in \A_2$ and likewise for morphisms.
For $\eps \in \{ +1, -1 \}$, we write \[
\M^\eps := \begin{cases}
\M \text{ as an $\A_1$-$\A_2$-bimodule category}, & \text{ if }\eps=+1,	\\
\overline{\M} \text{ as an $\overline{\A_2}$-$\overline{\A_1}$-bimodule category}, & \text{ if }\eps=-1.	
\end{cases}						\]
If $\M = \lmod{K}$ for an $H_1$-$H_2$-bicomodule algebra $K$, then $\overline{\M} \cong \lmod{\overline{K}}$ canonically as $(\lmod{H_1})$-$(\lmod{H_2})$-bimodule categories, where $\overline{K} := K^{\opp}$ is the opposite algebra with respect to $K$ considered as an $\overline{H_2}$-$\overline{H_1}$-bicomodule algebra.
For $M \in \lmod{K}$, the corresponding object $\overline{M}$ in $\lmod{\overline{K}}$ is given by the vector space dual $\Hom{\kk}{M,\kk}$ of $M$ with the natural induced $\overline{K}$-action.
For $\eps \in \{+1,-1\}$, we also write $K^\eps := \overline{K}$ if $\eps = -1$, and $K^\eps := K$ if $\eps=+1$. \\

\noindent A boundary circle of an oriented surface with defect lines labeled by bimodule categories gives rise to the following data.
Consider an oriented circle with $n$ marked points $(e_i)_{i \in \ZZ_n}$ that are each labelled with a sign $\eps_i \in \{ +1, -1 \}$, so that we call these points \emph{oriented}.
Label each segment between two marked points $e_i$ and $e_{i+1}$ by a finite pivotal tensor category $\A_{i,i+1}$ and label each marked point $e_i$ with a finite bimodule category $\M_i$, which is an $\A_{i-1,i}$-$\A_{i,i+1}$-bimodule category if $\eps_i = +1$, and an $\A_{i+1,i}$-$\A_{i,i-1}$-bimodule category if $\eps_i = -1$.
In other words, then $\M_i^{\eps_i}$ is an $\A_{i-1,i}^{\eps_i}$-$\A_{i,i+1}^{\eps_i}$-bimodule category, using the notation we have introduced above for opposite tensor categories and opposite bimodule categories.
The set $(\M_i^{\eps_i})_{i \in \ZZ_n}$ is called a \emph{string of cyclically composable bimodule categories}, according to \cite{fss}.

To this decorated circle with marked points, by the prescription of \cite{fss}, one associates a linear category, which we will explain now, see Definition \ref{def:gluing-category}.
First we consider the Deligne product $\M_1^{\eps_1} \boxt \cdots \boxt \M_n^{\eps_n}$ of the categories $(\M_i^{\eps_i})_{i \in \ZZ_n}$.
Following the above notation, corresponding to each segment between two marked points $e_i$ and $e_{i+1}$ in the circle there is the structure of an $\A_{i,i+1}^{\eps_{i+1}}$-$\A_{i,i+1}^{\eps_i}$-bimodule category on this Deligne product.
These $n$ bimodule category structures on the Deligne product commute with each other (up to canonical  coherent isomorphisms), since they act either on different Deligne factors or on two different sides of one of the bimodule categories.

For each of these bimodule category structures on the Deligne product we can consider so-called \emph{balancings}; e.g.\ for a $\boxt$-factorized object $(\opj{m_1}{\eps_1} \boxt \cdots \boxt \opj{m_n}{\eps_n})$ these are natural isomorphisms $( \opj{m_1}{\eps_1} \boxt \cdots \boxt \opj{m_i}{\eps_i} \boxt (\opj{a^{\eps_{i+1}}}{\eps_{i+1}} \lact \opj{m_{i+1}}{\eps_{i+1}}) \boxt \cdots \boxt \opj{m_n}{\eps_n} \lto \opj{m_1}{\eps_1} \boxt \cdots \boxt (\opj{m_i}{\eps_i} \ract \opj{a^{\eps_i}}{\eps_i}) \boxt \opj{m_{i+1}}{\eps_{i+1}} \boxt \cdots \boxt \opj{m_n}{\eps_n} )_{a \in \A_{i,i+1}}$
Here, for any category $\X$ and $\eps \in \{ +1, -1 \}$, we use the notation \[
\opj{x}{\eps} := \begin{cases}
x \in \X, & \text{ if }\eps=+1,	\\
\overline{x} \in \overline{\X}, & \text{ if }\eps=-1.	
\end{cases}						\]
for the object in $\X^\eps$ that corresponds to the object $x \in \X$, and for a pivotal tensor category $\A$ we use the notation \[
a^\eps := \begin{cases}
a , & \text{ if }\eps=+1,	\\
a^\vee , & \text{ if }\eps=-1.	
\end{cases}						\]
(While this notation would make sense for any tensor category that is not necessarily pivotal, it would be unnatural as it would arguably favor the right dual functor over all other odd-fold duals.
Therefore we assume that $\A$ is pivotal, which is the case of our interest anyway.)

Let us recall the general definition of such balancings for bimodule categories.

\begin{definition} \label{def:category-of-balancings}
Let $\A$ be a pivotal tensor category, let $\eps, \eps' \in \{ +1, -1 \}$ and let $\M$ be an $\A^\eps$-$\A^{\eps'}$-bimodule category.

Then the \emph{category $\balcat{\M}{\eps}{\eps'}$ of
	balancings in $\M$} has as objects pairs $(m,\beta)$, where $m$ is an object of $\M$ and the \emph{balancing} $(\beta_a : \opj{\dual{a}{\eps}}{\eps} \lact m \lsimto m \ract \opj{\dual{a}{\eps'}}{\eps'})_{a\in\A}$ is a natural isomorphism satisfying 
	\begin{center}
		\begin{tikzpicture}
		\matrix (m) [matrix of math nodes,row sep=2em,column sep=2em,minimum width=2em]{
			\opj{\dual{(a \ot b)}{\eps}}{\eps} \lact m \cong \opj{\dual{a}{\eps}}{\eps} \lact \opj{\dual{b}{\eps}}{\eps} \lact m & & \\
			& & \opj{\dual{a}{\eps}}{\eps} \lact m \ract \opj{\dual{b}{\eps'}}{\eps'} \\
			m \ract \opj{\dual{(a \ot b)}{\eps'}}{\eps'} \cong m \ract \opj{\dual{a}{\eps'}}{\eps'} \ract \opj{\dual{b}{\eps'}}{\eps'} & & \\
		};
		\path[-stealth]
		(m-1-1) edge node[above, xshift=20pt] {$\id{\opj{\dual{a}{\eps}}{\eps}} \lact \beta_{b}$} (m-2-3)
		edge node[left] {$\beta_{a \ot b}$} (m-3-1)
		(m-2-3) edge node[below, xshift=20pt] {$\beta_{a} \ract \id{\opj{\dual{b}{\eps'}}{\eps'}}$} (m-3-1)
		;
		\end{tikzpicture} \ \ \ \ \ \ \ \ \ 
		\begin{tikzpicture}
		\matrix (m) [matrix of math nodes,row sep=2em,column sep=2em,minimum width=2em]{
			\opj{\dual{\II}{\eps}}{\eps} \lact m & & \\
			& & m \\
			m \ract \opj{\dual{\II}{\eps'}}{\eps'} & & \\
		};
		\path[-stealth]
		(m-1-1) edge node[above] {$\cong$} (m-2-3)
		edge node[left] {$\beta_{\II}$} (m-3-1)
		(m-2-3) edge node[below] {$\cong$} (m-3-1)
		;
		\end{tikzpicture}
	\end{center}
	or, in formulas,
	\begin{align} \beta_{a \ot b} &= (\beta_{a} \ract \id{\opj{\dual{b}{\eps'}}{\eps'}}) \circ (\id{\opj{\dual{a}{\eps}}{\eps}} \lact \beta_{b}) \qquad \forall a,b \in \A,	\label{eq:hexagon-for-balancing}	\\
	\beta_{\II} &= \id{m} ,	\label{eq:triangle-for-balancing}
	\end{align}
	where we have omitted the bimodule constraint isomorphisms.
	
	The morphisms in the category of balancings are defined to be the morphisms in $\M$ that are compatible with the balancings.
\end{definition}
\begin{remark}
While this definition does not require any pivotal structure on the tensor category --  one can consider every dual to be the right dual, for example -- we will consider it only for a pivotal tensor category, since otherwise it would not coincide with the definition of the category of $\kappa$-balancings from \cite{fss} for an integer $\kappa \in \ZZ$.
In the construction in \cite{fss} this integer comes from a framing of the underlying surface and determines which of the various multiples of the double-dual functor, which are trivialised by a pivotal structure, we would need to insert in the above definition.
\end{remark}

The category that one finally assigns to the decorated circle with marked points, according to the prescription of \cite{fss} is as follows:

\begin{definition}[c.f.\ Definition~3.4 in \cite{fss}] \label{def:gluing-category}
Let $\mathbb{L}$ be an oriented circle with marked oriented points $\{ e_i \}_{i\in \ZZ_n}$ labelled by bimodule categories -- giving rise to a string $(\M_i^{\eps_i})_{i\in\ZZ_n}$ of cyclically composable bimodule categories.
The \emph{category $\textup{T}(\mathbb{L})$ assigned to the circle $\mathbb{L}$} is the category of balancings on the Deligne product $(\boxt_{i \in \ZZ_n} \M_i^{\eps_i})$ with respect to the $\A_{i,i+1}^{\eps_{i+1}}$-$\A_{i,i+1}^{\eps_i}$-bimodule category structures for all $i \in \ZZ_n$.
In formulas, \begin{equation}	\label{eq:definition-gluing-category}
\textup{T}(\mathbb{L}) := \balcat{\cdots \balcat{\boxt_{i \in \ZZ_n} \M_i^{\eps_i}}{\eps_2}{\eps_1}}{\eps_1}{\eps_n} .
\end{equation}
\end{definition}
\begin{remarks}~ \begin{itemize}
\item
This category is well-defined because the bimodule category structures on the Deligne product, with respect to which the balancings are considered, all commute with each other (up to canonical coherent natural isomorphisms).
In \cite{fss} it is explained that the category of balancings is monadic and that the monads for the balancings for the different bimodule category structures on the Deligne product satisfy a distributivity law, which also shows that \eqref{eq:definition-gluing-category} does not depend on the order in which we consider the balancings.
\item
The category assigned to a decorated circle with marked points reduces to the well-known Drinfeld center $\catcenter(\A)$, as shown in \cite{fss}, if all bimodule categories $\M_i$ are given by a single tensor category $\A$.
\end{itemize}
\end{remarks}

In Theorem \ref{thm:gluing-category-as-representation-category} we want to give a realization of such a category assigned to a decorated circle with marked points, in terms of representations of a $\kk$-algebra, namely the vertex algebra $C_v$, if the bimodule categories $(\M_i)_i$ are the representation categories of bicomodule algebras $(K_e)_{e\in\halfedges{v}}$.

To this end, we first show generally that the category of balancings, as in Definition \ref{def:category-of-balancings}, can be realized in such a representation-theoretic way.
For this, let $H$ be a
finite-dimensional Hopf algebra over $\kk$, let $\eps, \eps' \in \{+1, -1\}$ and let $K$ be an $H^\eps$-$H^{\eps'}$-bicomodule algebra.
Recall from Subsubsection \ref{subsubsec:tannaka-duality} that the category $\lmod{K}$ is an $H^\eps$-$H^{\eps'}$-bimodule category, so that we can consider the category of balancings $\balcat{\lmod{K}}{\eps}{\eps'}$ as defined in Definition \ref{def:category-of-balancings}.
On the other hand, recall from Definition \ref{def:balancing-algebra-hopf} the so-called balancing algebra $H^*_{\eps,\eps'}$, which is an $((H^{\eps'})^\coopp \ot H^\eps)$-module algebra, and recall from Definition \ref{def:crossed-product-algebra} the crossed product algebra $\vtxal{H^*_{\eps,\eps'}}{K}$, for which we consider $K$ as an $((H^{\eps'})^\coopp \ot H^\eps)$-comodule algebra.
This $\kk$-algebra $\vtxal{H^*_{\eps,\eps'}}{K}$ with underlying vector space $H^* \ot K$ is characterized by having $H^*$ and $K$ as subalgebras, and by the following instance of the straightening formula for the multiplication of an element $f \in H^*$ with an element $k \in K$:
\begin{equation} \label{eq:straightening-formula-most-general-balancing}
k \cdot f = f(\anti{k_{(1)}}{-\eps'} \cdot ? \cdot \anti{k_{(-1)}}{\eps}) \cdot k_{(0)}
\end{equation}

The following proposition proves that the category of balancings on $\lmod{K}$ is isomorphic to the representation category of the $\kk$-algebra $\vtxal{H^*_{\eps,\eps'}}{K}$.
This justifies the name ``balancing algebra'' for $H^*_{\eps,\eps'}$ and will be used in Theorem \ref{thm:gluing-category-as-representation-category}
to establish a connection between the vertex algebras defined in this paper and the categories assigned to circles in \cite{fss}.

\begin{proposition} \label{prop:category-of-balancings-as-representation-category}
Let $H$ be a semisimple
finite-dimensional Hopf algebra over $\kk$, let $\eps, \eps' \in \{+1, -1\}$ and let $K$ be an $H^\eps$-$H^{\eps'}$-bicomodule algebra.
Then there is a canonical equivalence of $\kk$-linear categories
\[ \balcat{\lmod{K}}{\eps}{\eps'} \cong \lmod{(H_{\eps,\eps'}^* \ogreaterthan K)}. \] 
\end{proposition}
\begin{proof}
Let $(M, \beta = (\beta_{X} : \opmod{\dual{X}{\eps}}{\eps} \lact M \lsimto M \ract \opmod{\dual{X}{\eps'}}{\eps'})_{X\in \lmod{H}})$ be an object in $\balcat{\lmod{K}}{\eps}{\eps'}$.
Recall that the vector spaces underlying the modules $\opmod{X^\eps}{\eps} \in \lmod{H^\eps}$ and $\opmod{X^{\eps'}}{\eps'} \in \lmod{H^{\eps'}}$ are the same as $X \in \lmod{H}$.
In this proof, to simplify notation, we will often write $\beta_{X}$ as a map $X \ot M \lto M \ot X$, keeping implicit the module structures on the respective vector spaces.

We define, using $\beta$, a left $H^*$-module structure on $M$ as follows.
We denote by $H_\reg \in \lmod{H}$ the left regular $H$-module with underlying vector space $H$, whose $H$-action is defined by left multiplication.
\begin{align} \label{eq:H^*-action-in-terms-of-balancing}
\rho : H^* \ot M &\lto M, \\
f \ot m &\lmapsto (\id{M} \ot f)\beta_{H_\reg}(1_H \ot m) \nonumber
\end{align}
We show that this indeed satisfies the axioms of a left $H^*$-module:
On the one hand we have, for $f,g \in H^*$ and $m \in M$,
\begin{align*}
\rho(f \ot \rho(g \ot m))
							&\eqbydef (\id{M} \ot f)\beta_{H_\reg}(1_H \ot (\id{M} \ot g)\beta_{H_\reg}(1_H \ot m))	\\
							&= (\id{M} \ot f \ot g)(\beta_{H_\reg} \ot \id{H})(\id{H} \ot \beta_{H_\reg})(1_H \ot 1_H \ot m) .
\end{align*}
On the other hand, we have
\begin{align*}
\rho((f\cdot g) \ot m) &= (\id{M} \ot (f\cdot g))\beta_{H_\reg}(1_H \ot m)	\\
								&= (\id{M} \ot f\ot g)(\id{M} \ot \Delta)\beta_{H_\reg}(1_H \ot m)	\\
								&\stackrel{\beta\text{ natural}}{=} (\id{M} \ot f\ot g)\beta_{H_\reg \ot H_\reg}(\Delta(1_H) \ot m)	\\
								&= (\id{M} \ot f\ot g)\beta_{H_\reg \ot H_\reg}(1_H \ot 1_H \ot m)	\\
&\stackrel{\eqref{eq:hexagon-for-balancing}}{=} (\id{M} \ot f \ot g)(\beta_{H_\reg} \ot \id{H})(\id{H} \ot \beta_{H_\reg})(1_H \ot 1_H \ot m) ,
\end{align*}
where we use in the third line that the coproduct of $H$ is an $H$-module morphism $\Delta : H_\reg \lto H_\reg \ot H_\reg$.
This shows one of the two axioms of an $H^*$-module.
For the other axiom, let again $m \in M$.
Then, indeed, we have
\begin{align*}
\rho(1_{H^*} \ot m) &= \rho(\eps \ot m)										\\
&\eqbydef (\id{M} \ot \eps)\beta_{H_\reg}(1_H \ot m)				\\
&\stackrel{\beta\text{ natural}}{=} \beta_\kk (\eps(1_H) \ot m) 	\\
&= m,
\end{align*}
where we use in the third line that the co-unit of $H$ is an $H$-module morphism $\eps : H_\reg \lto \kk$.
Hence, we have shown that $\rho$ endows $M$ with the structure of an $H^*$-module.

To prove that $(M, \rho)$ is an object of $\lmod{(\vtxal{H^*_{\eps,\eps'}}{K})}$ we have to show that the just defined $H^*$-action $\rho$ and the given $K$-action on $M$, which we simply denote by $K \ot M \to M, k \ot m \mapsto k.m$, satisfy the straightening formula \eqref{eq:straightening-formula-most-general-balancing}.
That is, we have to show that, for all $f \in H^*$, $k \in K$ and $m \in M$,
\begin{equation}
k . ((\id{M} \ot f)\beta_{H_\reg}(1_H \ot m)) = (\id{M} \ot f(\anti{k_{(1)}}{-\eps'} \cdot ? \cdot \anti{k_{(-1)}}{\eps}) )\beta_{H_\reg}(1_H \ot k_{(0)}.m)
\end{equation}
We start with the right-hand side:
\begin{align*}
(\id{M} \ot f(\anti{k_{(1)}}{-\eps'} \cdot ? \cdot \anti{k_{(-1)}}{\eps}) )\beta_{H_\reg}(1_H \ot k_{(0)}.m)	
&\stackrel{\beta\text{ natural}}{=} (\id{M} \ot f(\anti{k_{(1)}}{-\eps'} \cdot ? ))\beta_{H_\reg}(\anti{k_{(-1)}}{\eps} \ot k_{(0)}.m)	\\
&\stackrel{\beta_{H_\reg}\text{ $K$-linear}}{=} ((k_{(0)} . ?) \ot f(\anti{k_{(2)}}{-\eps'} \anti{k_{(1)}}{\eps'} \cdot ? )) \beta_{H_\reg} (1_H \ot m)	\\
&= k.((\id{M} \ot f)\beta_{H_\reg}(1_H \ot m)).
\end{align*}
Here we use in the first line that right multiplication by any element $h\in H$ is an $H$-module morphism $(? \cdot h) : H_\reg \lto H_\reg$ for the left regular $H$-module $H_\reg$, and in the last line we use the defining property of the antipode of $H$.
This concludes the proof that $(M, \rho) \in \lmod{(\vtxal{H^*_{\eps,\eps'}}{K})}$. \\ \\

\noindent Conversely, assume that $M \in \lmod{(\vtxal{H^*_{\eps,\eps'}}{K})}$ and let us define on $M$ a balancing $\beta_X : \opmod{\dual{X}{\eps}}{\eps} \lact M \lto M \ract \opmod{\dual{X}{\eps'}}{\eps'}$ for all $X \in \lmod{H}$.
Denoting by $(e^i \in H^*)_i$ and $(e_i \in H)_i$ a pair of dual bases, we define
\begin{align*}
\beta_X : X \ot M &\lto M \ot X,	\\
x \ot m &\lmapsto \sum_i e^i.m \ot e_i.x,
\end{align*}
where $e_i . x$ refers to $X$ as an $H$-module, not $\opmod{X^{\eps'}}{\eps'}$ as an $H^{\eps'}$-module, even though we will show that $\beta_X$ is a $K$-module morphism $\opmod{\dual{X}{\eps}}{\eps} \lact M \lto M \ract \opmod{\dual{X}{\eps'}}{\eps'}$.
Indeed, for $k \in K, x \in X, m \in M$, we calculate
\begin{align*}
k . ( \beta_X( x \ot m) &\eqbydef \sum_i (k_{(0)} . e^i . m) \ot (\anti{k_{(1)}}{\eps'} . e_i . x)	\\
									&\stackrel{\eqref{eq:straightening-formula-most-general-balancing}}{=} \sum_i (e^i(\anti{k_{(1)}}{-\eps'} \cdot ? \cdot \anti{k_{(-1)}}{\eps}) . k_{(0)} . m) \ot (\anti{k_{(2)}}{\eps'} . e_i . x)	\\
									&= \sum_i (e^i . k_{(0)} . m) \ot (\anti{k_{(2)}}{\eps'} . \anti{k_{(1)}}{-\eps'} . e_i . \anti{k_{(-1)}}{\eps} . x)	\\
									&= \sum_i (e^i . k_{(0)} . m) \ot (e_i . \anti{k_{(-1)}}{\eps} . x)	\\
									&\eqbydef \beta_X( k. (x \ot m) )
\end{align*}
Furthermore, it can be seen directly that $(\beta_X)_{X \in \lmod{H}}$ is a natural family.
Indeed, for any $H$-module morphism $f : X \lto Y$ and $x \in X, m \in M$, we have $\beta_Y(f(x) \ot m) \eqbydef \sum_i e^i.m \ot e_i.(f(x)) =  \sum_i e^i.m \ot f(e_i.x) \eqbydef (\id{M}\ot f)\beta_X(x \ot m)$.

It remains to show that $(\beta_X)_{X \in \lmod{H}}$ satisfies axioms \eqref{eq:hexagon-for-balancing} and \eqref{eq:triangle-for-balancing}, i.e.\ $\beta_{X \ot Y} = (\beta_X \ot \id{Y})(\id{X} \ot \beta_Y)$ for all $X, Y \in \lmod{H}$ and $\beta_\kk = \id{M}$.

For the first identity, let $x\in X$, $y \in Y$ and $m \in M$.
Then on the one hand we have $\beta_{X \ot Y}(x \ot y \ot m) \eqbydef \sum_i e^i.m \ot e_i.(x \ot y) = \sum_i e^i.m \ot ({e_i}_{(1)}.x) \ot ({e_i}_{(2)}.y)$.
On the other hand, $(\beta_X \ot \id{Y})(\id{X} \ot \beta_Y)(x\ot y \ot m) \eqbydef \sum_{i, j} e^j . e^i . m \ot e_j . x \ot e_i . y = \sum_i e^i.m \ot ({e_i}_{(1)}.x) \ot ({e_i}_{(2)}.y),$ where the last identity uses that the multiplication of $H^*$ is defined as the dual of the co-multiplication of $H$.

In order to show \eqref{eq:triangle-for-balancing}, we use that the unit of $H^*$ is the co-unit $\eps : H \to \kk$ of $H$.
For $\lambda \in \kk$ and $m \in M$ we thus have $\beta_\kk (m \ot \lambda) \eqbydef \sum_i e^i.m \ot \eps(e_i) \lambda = 1_{H^*} . m = m$.	\\	\\
So far in this proof, we have shown that on $M \in \lmod{K}$ one can construct out of a balancing on $M$ an $H^*$-action such that $M$ becomes an $(\vtxal{H^*_{\eps,\eps'}}{K})$-module, and that conversely out of an $(\vtxal{H^*_{\eps,\eps'}}{K})$-module structure one can construct a balancing on $M \in \lmod{K}$.
To conclude the proof of the proposition we have to show that these two assignments are inverse to each other.

First, assume that $(M,\beta) \in \balcat{\lmod{K}}{\eps}{\eps'}$.
Consider the balancing $\beta'$ on $M$ that is constructed from the $H^*$-action on $M$ which in turn is constructed from $\beta$, as shown above.
For $X\in \lmod{H}$, $x\in X$ and $m \in M$ we have
\begin{align*}
\beta'_X(x \ot m) &\eqbydef \sum_i (\id{M} \ot e^i)\beta_{H_\reg}(1_H \ot m) \ot e_i.x	\\
						&= (\beta_{H_\reg}(1_H \ot m))_{(M)} \ot (\beta_{H_\reg}(1_H \ot m))_{(X)}.x	\\
						&\stackrel{\beta\text{ natural}}{=} \beta_{X}(x \ot m),
\end{align*}
where use the notation $(\beta_{H_\reg}(1_H \ot m))_{(M)} \ot (\beta_{H_\reg}(1_H \ot m))_{(X)} := \beta_{H_\reg}(1_H \ot m) \in M \ot X$, and in the third line we use that $(?.x) : H_\reg \lto X$ is an $H$-module morphism for any $x \in X$.

Finally, assume that $M \in \lmod{(\vtxal{H^*_{\eps,\eps'}}{K})}$ with $H^*$-action $\rho : H^* \ot M \lto M$.
Consider the $H^*$-action $\rho'$ on $M$ that is constructed from the balancing on $M$ which in turn is constructed from $\rho$, as shown above.
For $f \in H^*$ and $m \in M$ we then have
\begin{align*}
\rho'(f \ot m) \eqbydef \sum_i (\id{M} \ot f) (\rho(e^i \ot m) \ot e_i.1_H)
					= \sum_i \rho(e^i \ot m) f(e_i)
					= \rho(f \ot m),
\end{align*}
which concludes the proof of the proposition.
\end{proof}

Now, finally, we can prove the main result of this appendix.
Most of the work for this has already been done in the proof of Proposition \ref{prop:category-of-balancings-as-representation-category}.
Let $v \in \Sigma^0$ be a vertex of a labeled cell decomposition of $\Sigma$ so that $(K_e)_{e\in\halfedges{v}}$ are bicomodule algebras labelling the incident edges at $v$.
Let $\mathbb{L}_v$ be the corresponding circle with marked points which are labeled by cyclically composable bimodule categories $(\lmod{K_e})_{e\in\halfedges{v}}$.

\begin{theorem} \label{thm:gluing-category-as-representation-category}
Let $v\in\Sigma^0$ be a vertex in a labelled (as defined in Definition \ref{def:bicomodule-algebra}) cell decomposition of a compact oriented surface $\Sigma$.
There is a canonical equivalence of $\kk$-linear categories
\begin{equation*}
\textup{T}(\mathbb{L}_v) \cong \lmod{C_v}
\end{equation*}
between the category assigned by the modular functor $\textup{T}$, constructed in \cite{fss}, to the circle $\mathbb{L}_v$ with marked points corresponding to the half-edges incident to a vertex $v \in \Sigma^0$ and the representation category of the algebra $C_v$.
\end{theorem}
\begin{proof}
Consider the bicomodule algebra $(\bigot_{e \in \halfedges{v}} K_e^{\vertexedge{v}{e}})$, which realizes the Deligne product $\boxt_{e\in\halfedges{v}} (\lmod{K_e})^{\vertexedge{v}{e}} = \lmod{(\bigot_{e \in \halfedges{v}} K_e^{\vertexedge{v}{e}})}$ as a representation category.
For each incident site $p \in \vtxsites{v}$, which corresponds to a segment between two marked points of the corresponding decorated circle $\mathbb{L}_v$ and is labeled by a Hopf algebra $H_p$, it has an $H_p^\vertexedge{v}{e_p}$-$H_p^{\vertexedge{v}{e'_p}}$-bicomodule structure, where $e_p$ and $e'_p$ are half-edges incident to $v$ in the boundary of the plaquette $p$, cf.\ Figure \ref{fig:site}.
Denote the sites in $\vtxsites{v}$ in clockwise order around $v$ by $(p_{1,2}, \dots, p_{n,1})$ and abbreviate $\vertexedge{v}{e_{p_{i,i+1}}} =: \eps_{i+1}$ and $\vertexedge{v}{e'_{p_{i,i+1}}} =: \eps_i$.
We then repeatedly apply Proposition \ref{prop:category-of-balancings-as-representation-category} for each of these $H_{p_{i,i+1}}^{\eps_{i+1}}$-$H_{p_{i,i+1}}^{\eps_i}$-bicomodule structures.
This is well-defined and does not depend on the order, since for different $p \in \vtxsites{v}$ the bicomodule structures commute with each other.
We hence obtain an equivalence of categories
\begin{align*}
\balcat{\cdots \balcat{\boxt_{e \in \halfedges{v}} (\lmod{K_e})^{\vertexedge{v}{e}}}{\eps_2}{\eps_1}}{\eps_1}{\eps_n} &\cong \lmod{\big( \vtxal{ \big( (H_{p_{n,1}})_{\eps_1,\eps_n}^*\ot\cdots\ot (H_{p_{1,2}})_{\eps_2,\eps_1}^* \big) }{(\bigot_{e \in \halfedges{v}} K_e^{\vertexedge{v}{e}})} \big) }	\\
&\eqbydef \lmod{C_v}	,
\end{align*}
which concludes the proof.
\end{proof}
\begin{remark}
Since the category of balancings reduces to the Drinfeld center $\catcenter(\A)$ if all bimodule categories $\M_i$ are given by a single tensor category $\A$, as shown in \cite{fss}, we see that also in our construction in case of only transparently labeled edges incident to the vertex $v$, the category of labels is the representation category of the Drinfeld double, just as in the Kitaev construction without defects, see e.g.\ \cite{balKir}.
\end{remark}

\pagebreak

\end{document}